\documentclass{amsart}
\usepackage[utf8]{inputenc}
\usepackage[table]{xcolor}
\usepackage{mathtools,tikz-cd,amsmath,amssymb}
\usepackage[shortlabels]{enumitem}
\usepackage[parfill]{parskip}

\usepackage{amsthm,amssymb,latexsym,epic,bbm,comment,mathrsfs}
\usepackage[all,2cell]{xy}
\xyoption{2cell}

\allowdisplaybreaks

\newtheorem{theorem}{Theorem}[section] 
\newtheorem{conjecture}[theorem]{Conjecture}

\newtheorem{problem}[theorem]{Problem}
\newtheorem{lemma}[theorem]{Lemma}
\newtheorem{corollary}[theorem]{Corollary}

\newtheorem{proposition}[theorem]{Proposition}

\newtheorem{thmx}{Theorem}
\theoremstyle{definition}
\newtheorem{example}[theorem]{Example}

\newtheorem{remark}[theorem]{Remark}

\newcommand{\tto}{\twoheadrightarrow}
\font\sc=rsfs10

\newcommand{\cL}{\sc\mbox{L}\hspace{1.0pt}}
\newcommand{\cP}{\sc\mbox{P}\hspace{1.0pt}}

\font\scc=rsfs7

\newcommand{\ccL}{\scc\mbox{L}\hspace{1.0pt}}

\newcommand{\Hom}{\text{\rm Hom}}

\newcommand{\ext}{\operatorname{ext}}

\newcommand{\Z}{\mathbb Z}
\newcommand{\inv}{^{-1}}

\newcommand{\cO}{\mathcal O}

\newcommand{\fg}{{\mathfrak{g}}}

\newcommand{\ug}{U(\mathfrak{g})}

\newcommand{\soc}{\operatorname{soc}}

\renewcommand{\mod}{\operatorname{mod}}

\newcommand{\res}{\operatorname{Res}}

\newcommand{\Add}{\operatorname{add}}

\newcommand{\surj}{\rightarrow\mathrel{\mkern-14mu}\rightarrow}

\newcommand{\inj}{\hookrightarrow}

\newcommand{\half}{\frac{1}{2}}

\newcommand{\kl}{\underline H}

\newcommand{\hcii}{{}_{\hspace{4pt}0}^\infty\mathcal H^I_0}
\newcommand{\hciii}{\ \!_0^\infty\!\mathcal H^\infty_0}
\newcommand{\hci}{\ \!^\infty_0\!\mathcal H^1_0}
\newcommand{\ann}{\operatorname{Ann}}

\usepackage{todonotes}

\begin{document}

\title[Some homological properties of category $\mathcal{O}$, V]
{Some homological properties of category $\mathcal{O}$, V}

\author[H.~Ko, V.~Mazorchuk and R.~Mr{\dj}en]
{Hankyung Ko, Volodymyr Mazorchuk and Rafael Mr{\dj}en}

\begin{abstract}
We compute projective dimension of translated simple
modules in the regular block of the BGG category $\mathcal{O}$
in terms of Kazhdan-Lusztig combinatorics. This
allows us to determine which projectives can appear at
the last step of a minimal projective resolution for
a translated simple module, confirming a conjecture by
Johan K{\aa}hrstr{\"o}m. We also derive some inequalities,
in terms of Lusztig's $\mathbf{a}$-function,
for possible degrees in which the top (or socle) of
a translated simple module can live. Finally, we
relate Kostant's problem with decomposability and
isomorphism of translated simple modules, addressing 
yet another conjecture by Johan K{\aa}hrstr{\"o}m.

\end{abstract}

\maketitle

\setcounter{tocdepth}{1}
\tableofcontents

\section{Introduction, motivation and description of the results}\label{s1}

\subsection{Setup}\label{s1.1}
Let $\mathfrak{g}$ be a semi-simple complex Lie algebra with a fixed
triangular decomposition
\begin{displaymath}
\mathfrak{g}=\mathfrak{n}_-\oplus \mathfrak{h}\oplus\mathfrak{n}_+
\end{displaymath}
and $U(\mathfrak{g})$ be the universal enveloping algebra of $\mathfrak{g}$.
Let $\mathcal{O}$ denote the associated Bernstein-Gelfand-Gelfand (BGG)
category $\mathcal{O}$, see \cite{BGG,Hu}. Let, further, 
$\mathcal{O}_0$ denote the {\em principal block} $\mathcal{O}$, that is,
the indecomposable direct summand containing the trivial $\mathfrak{g}$-module.

Let $W$ be the Weyl group of $\mathfrak{g}$. It acts naturally on 
$\mathfrak{h}^*$ via $(w,\lambda)\mapsto w(\lambda)$. We consider 
the {\em dot-action} $(w,\lambda)\mapsto w\cdot \lambda$ of $W$ which 
is obtained by shifting the usual action by the half of the sum of all
positive roots.

The category $\mathcal{O}_0$ is equivalent to $A$-mod, for some finite-dimensional
associative basic algebra $A$, unique up to isomorphism. 
Simple objects in $\mathcal{O}_0$ are exactly the 
(pairwise non-isomorphic) simple
highest weight modules $L_w:=L(w\cdot 0)$ of highest weight $w\cdot 0$, for
$w\in W$.
The category $\mathcal{O}_0$ is equipped with the action of the
monoidal category $\cP$ of {\em projective functors}, as defined in \cite{BG}.
Up to isomorphism, indecomposable projective functors are also in bijection with the elements 
in $W$, where the indecomposable projective functor $\theta_w$, for $w\in W$,
is normalized such that it sends the projective cover $P_e$ of
$L_e$ to the projective cover $P_w$ of $L_w$.

The algebra $A$ is Koszul by \cite{So}, in particular, it admits a
positive $\mathbb{Z}$-grading
. Denote by 
$\mathcal{O}^{\mathbb{Z}}_0$ the category of finite dimensional 
$\mathbb{Z}$-graded $A$-modules with morphisms being homogeneous homomorphisms 
of degree zero, see e.g., \cite{St}.

\subsection{Motivation}\label{s1.2}
The first major motivation for the present paper is the following: 

\begin{conjecture}[\cite{KM}]\label{conj1}
Assume that $\mathfrak{g}$ is of type $A$. Then, for $x,y\in W$, the module 
$\theta_x L_y$ is either indecomposable or zero.
\end{conjecture}

Various approaches to Conjecture~\ref{conj1} were considered in
\cite{KM} and \cite{CMZ}. In the latter paper,
the conjecture was confirmed in the cases 
$\mathfrak{g}=\mathfrak{sl}_n$, where $n=2,3,4,5,6$.
For other values of $n$, a number of special results are
obtained in \cite{KM} and \cite{CMZ}.

The second major motivation for the present paper
is the so-called {\em Kostant's problem}, as popularized 
in \cite{Jo}, the content of which is to determine all $w\in W$
for which the universal enveloping algebra surjects onto the space
of linear endomorphism of $L_w$ that are locally finite with respect to the
adjoint action (we will denote this property by $\mathbf{K}(w)$, see \S~\ref{s2.3}
for details). This problem was studied in 
\cite{Ma1,Ma2,MS2,Kh,KhM}, see also the references therein.
In March 2019, the second author
received an email from Johan K{\aa}hrstr{\"o}m with the 
following conjecture (based on extensive computer computation).

\begin{conjecture}[J.~K{\aa}hrstr{\"o}m]\label{kconj-intro}
For a Duflo element $d\in W$, the following assertions are equivalent:
\begin{enumerate}[$($a$)$]
\item\label{kconj-1-intro} $\mathbf{K}(d)$.
\item\label{kconj-2-intro} $\theta_x L_d\not\cong\theta_y L_d$, for all
$x\neq y\in W$ such that $\theta_x L_d\neq 0$ and $\theta_y L_d\neq 0$.
\item\label{kconj-3-intro} For all
$x\neq y\in W$ such that $\theta_x L_d\neq 0$ and $\theta_y L_d\neq 0$, 
there exists $z\in W$ and $i\in\mathbb{Z}$
such that $[\theta_x L_d:L_z\langle i\rangle]\neq [\theta_y L_d:L_z\langle i\rangle]$
in $\mathcal{O}_0^{\mathbb{Z}}$.
\item\label{kconj-4-intro} For all
$x\neq y\in W$ such that $\theta_x L_d\neq 0$ 
and $\theta_y L_d\neq 0$, there is $z\in W$ 
such that $[\theta_x L_d:L_z]\neq [\theta_y L_d:L_z]$ in $\mathcal{O}_0$.
\end{enumerate}
\end{conjecture}

The final piece of motivation 
for the present paper is the problem  to determine
the projective dimension for all modules of the form
$\theta_xL_y\in\mathcal{O}_0$, as formulated in 
\cite[Problem~24]{Ma2}. 
In connection to this problem,
the email from Johan K{\aa}hrstr{\"o}m mentioned above 
contained the following conjecture (also based on 
extensive computer computation):

\begin{conjecture}[J.~K{\aa}hrstr{\"o}m]\label{kconj2-intro}
Let $x,y\in W$ and $k$ be the projective dimension of
$\theta_xL_y$. Assume that $z\in W$ is such that $\mathrm{Ext}^k(\theta_xL_y,L_z)\neq 0$.
Then $z$ and $x$ belong to the same Kazhdan-Lusztig left cell.
\end{conjecture}

\subsection{Description of the results}\label{s1.3}

The first main result of the present paper, 
see Proposition \ref{prop8.5-1}, Formula~\eqref{pd=a+b},
Theorem~\ref{cong-kmm-2} and Corollary~\ref{cor:kconj-301}, is:

\begin{thmx}
Conjecture \ref{kconj2-intro} is true.
Moreover, the projective dimension of $\theta_xL_y$, for $x,y\in W$
such that $\theta_xL_y\neq 0$, is given by 
\[\mathbf{a}(w_0x)+\mathbf{b}(y\inv w_0,w_0x\inv),\]
where $w_0$ is the longest element in $W$, $\mathbf{a}$ denotes Lusztig's $\mathbf{a}$-function, and $\mathbf{b}(u,v)$ is defined as the maximal degree shift of the composition factors in $\theta_u L_v$ (see Subsection \ref{ss:b}).
\end{thmx}

Note that both $\mathbf{a}(w_0x)$ and $\mathbf{b}(y\inv w_0,w_0x\inv$) are ``combinatorial''
in the sense that they are defined explicitly in terms of Kazhdan-Lusztig polynomials.

Our second main result is the following
statment, see Theorem~\ref{theorem1},
which combines ingredients of Conjecture~\ref{conj1}
with parts of Conjecture~\ref{kconj-intro}:

\begin{thmx}\label{thm1}
For $y\in W$, the assertion $\mathbf{K}(y)$ is true if and only 
if both of the following conditions hold.
\begin{enumerate}
    \item for all $x\neq z$, we have $\theta_xL_y\not\cong \theta_zL_y$ whenever nonzero (we refer to this property as $\mathbf{Kh}(y)$);
    \item for each $x\in W$, the module $\theta_xL_y$ is either indecomposable or zero (we refer to this property as $\mathbf{KM}(*,y)$).
\end{enumerate}
In particular, if Conjecture \ref{conj1} is true, then \ref{kconj-1-intro}$\Leftrightarrow$\ref{kconj-2-intro} in Conjecture \ref{kconj-intro} holds in type A.
\end{thmx}

In Theorem~\ref{thm-kmm-original} we obtain some bounds
(given in terms of the $\mathbf{a}$-function) on the degrees of simple constituents in the top (or socle) of  $\theta_x L_y$. When the bound prescribed by Theorem~~\ref{thm-kmm-original} is achieved, an interesting question is the multiplicity of the corresponding simple constituent in the top. We formulate a condition on this multiplicity that we call $\mathbf{KMM}$, 
see \eqref{eqnnn23}, which would imply Conjecture~\ref{conj1} in type $A$.

We also relate Kostant's problem for different simple highest
weight modules from the same Kazhdan-Lusztig left cell,
see Corollary~\ref{of con1} and Proposition \ref{decyLy} for the condition  \ref{item:C_ii_d}:

\begin{thmx}\label{thm2}
Let $d\in W$ be a Duflo element and let $\mathcal{L}$ be the Kazhdan-Lusztig left cell containing $d$. 
\begin{enumerate}
    \item If $\mathbf{K}(d)$ is not true, then $\mathbf{K}(y)$ is not true, for all $y\in \mathcal{L}$;
    \item If $\mathbf{K}(d)$ is true, then, for each $y\in \mathcal{L}$, the following conditions are equivalent:
    \begin{enumerate}[$($a$)$]
        \item $\mathbf{K}(y)$;
        \item $\theta_{y}L_{y\inv}\cong \theta_d L_d$;
        \item $t_yt_{y\inv} = t_d$ in the asymptotic ring for $W$ (see \S \ref{ss:asym});
        \item $\mathbf{KM}(y,y\inv)$;
        \item\label{item:C_ii_d} if, furthermore, $\fg$ is of classical type: the $\mathtt{H}$-cell of $y$ and the $\mathtt{H}$-cell of $d$ contain the same number of elements (here, an $\mathtt{H}$-cell is the intersection of a Kazhdan-Lusztig right cell and  a Kazhdan-Lusztig left cells, see Subsection \ref{ss:cella}).
    \end{enumerate}
\end{enumerate}
\end{thmx}

Theorem \ref{thm1} and Theorem \ref{thm2} give a conjectural answer to Kostant's problem in terms of Kazhdan-Lusztig combinatorics as follows (see Corollary \ref{cor1.4}). 

\begin{corollary}\label{introcor}
Suppose \ref{kconj-1-intro} $\Leftrightarrow$ \ref{kconj-3-intro} in Conjecture \ref{kconj-intro} is true for $W$. Then, for each $y\in W$, the condition $\mathbf{K}(y)$ is equivalent to the conjunction of the following conditions:
\begin{enumerate}
    \item $t_yt_{y\inv}=t_d$ in the asymptotic ring for $W$;
    \item For the Duflo element $d\sim_\mathtt{L} y$ and $x,x'\leq_\mathtt{R} d$, if $h_{z,x,d}=h_{z,x',d}$, for all $z\in W$, then $x=x'$.
\end{enumerate}
\end{corollary}
Here $h_{z,x,d}\in\Z[v,v\inv]$ is the structure coefficient for the Kazhdan-Lusztig basis (see \eqref{hxyz}).

In the last section, we use Theorem \ref{thm1}, Theorem \ref{thm2}, Kazhdan-Lusztig combinatorics, and some other results in the paper to determine $\mathbf{K}(y)$, as well as  $\mathbf{KM}(y)$ and $\mathbf{Kh}(y)$, 
for all $y\in W$ in a number of small rank cases.  
The result confirms Conjecture \ref{kconj-intro} in type $A_n$ for $n\leq 5$ and in type $BCD$ in rank $\leq 4$.
In particular, we completely solve Kostant's problem  
in type $A_5$, see Corollary \ref{corollary:kostantA5}. 
This question was considered before in \cite{KhM, Kh}, where it was solved completely for
$A_n$, where $n\leq 4$, and a partial answer for $A_5$ was given.

\subsection*{Acknowledgments}
This research was partially supported by
the Swedish Research Council, 
G{\"o}ran Gustafsson Stiftelse and Vergstiftelsen. The third author was also partially supported by the QuantiXLie Center of Excellence grant no. KK.01.1.1.01.0004 funded by the European Regional Development Fund.

We are especially indebted to Johan K{\aa}hrstr{\"o}m who shared with us
his ideas which started the work on this paper.

\section{A zoo of questions about $\mathcal{O}_0$}\label{s2}

In this sections we both recall some classical open problems
and questions about $\mathcal{O}_0$ and propose some new ones. 
In the rest of the paper we look deeper into 
connection between these problems and questions.

For any function $\mathbf{F}:W\to \{\mathrm{false},\mathrm{true}\}$,
we write $\mathbf{F}(*)$ for the conjunction of all $\mathbf{F}(w)$,
where $w\in W$, and similarly for functions of several variables.
We also write $\mathbf{F}(*_d)$ for the conjunction of all $\mathbf{F}(w)$,
where $w\in W$ is a Duflo element.

\subsection{Indecomposability of translation of simple modules}\label{s2.1}

For $x,y\in W$, we denote by $\mathbf{KM}(x,y)$ the statement ``the module
$\theta_x L_y$ is either indecomposable or zero''. The following problem
is still open:

\begin{problem}\label{prob2}
Determine all $x,y\in W$, for which $\mathbf{KM}(x,y)$ is true.
\end{problem}

Conjecture~\ref{conj1} asserts that $\mathbf{KM}(x,y)$ is always true in type $A$.
We note that $\theta_x L_y$ is non-zero if and only if $x^{-1}\leq_\mathtt{L} y$, where
$\leq_\mathtt{L}$ denotes the Kazhdan-Lusztig left order on $W$ from \cite{KL}, 
see, for example, \cite[Formula~(1)]{KM} and the references preceding this formula.
If $\mathfrak{g}$ is of type $B_2$ and $1,2$ are the two simple reflections in 
$W$, then $\mathbf{KM}(12,21)$ is known to be false, see \cite[Subsection~5.1]{KM}.

\subsection{Graded simple tops of translated simple modules}\label{s2.2}

Denote by $\langle 1\rangle$ the grading shift on
$\mathcal{O}_0^{\mathbb{Z}}$ normalized such that
it maps degree one to degree zero. Fix standard 
graded lifts $L_w$, where $w\in W$,
of simple modules concentrated in degree zero. Fix standard graded lifts $P_w$, 
where $w\in W$, of indecomposable projective modules such that their tops are
concentrated in degree zero.
According to \cite{St}, each $\theta_w$, where $w\in W$, also has a graded lift
(unique up to isomorphism and shift of grading),
which we normalize such that $\theta_w P_e\cong P_w$ holds as graded modules.
Morphisms in $\mathcal{O}^{\mathbb{Z}}$ is denoted by $\mathrm{hom}$.

We will use Lusztig's $\mathbf{a}$-function 
$\mathbf{a}:W\to\mathbb{Z}_{\geq 0}$ 
from \cite{Lu1,Lu2} (see \S \ref{ss:cella} for details). One of our principal observations 
in this paper is the following:

\begin{theorem}\label{thm-kmm-original}
For $x,y,z\in W$ and $i\in\mathbb{Z}$, the condition
\begin{equation}\label{eq8}
\mathrm{hom}(\theta_x L_y,L_z\langle i \rangle)\neq 0
\quad\text{ implies }\quad i\geq \mathbf{a}(x).
\end{equation}
\end{theorem}

Theorem~\ref{thm-kmm-original} is proved in Section~\ref{s7}.

Recall that each Kazhdan-Lusztig left (and right) cell contains a
unique distinguished involution, called the {\em Duflo element}.
In type $A$, all involutions are Duflo elements.
We propose the following:

\begin{conjecture}\label{conj3prime}
Let $d,y\in W$ be such that $d$
is a Duflo element. Let, further, $M$ be an indecomposable summand of $\theta_d L_y$. Then 
\begin{equation}\label{eq8prime-2}
\dim \mathrm{hom}(M,L_y\langle \mathbf{a}(d) \rangle)=1.
\end{equation}
\end{conjecture}

Let $d\in W$ be a Duflo element. Then there is a
(unique up to a non-zero scalar and homogeneous of degree zero) non-zero 
natural transformation $\zeta:\theta_d\to\theta_e\langle\mathbf{a}(d)\rangle$,
see \cite[Section~7]{MM3}.
If $y\in W$ is such that $\theta_d L_y\neq 0$, then $\zeta_{L_y}$ is
non-zero because the cokernel of $\zeta$ is killed by $\theta_d$,
see \cite[Proposition~17]{MM1},
and hence this cokernel must annihilate $L_y$. Therefore,
$\dim \mathrm{hom}(\theta_d L_y,L_y\langle \mathbf{a}(d) \rangle)\geq 1$.

For fixed $x,y\in W$, we denote by $\mathbf{KMM}(x,y)$
the property
\begin{equation}\label{eqnnn23}
\dim \mathrm{hom}(\theta_x L_y,L_y\langle \mathbf{a}(x) \rangle)\leq 1. 
\end{equation}

\subsection{Kostant's problem}\label{s2.3}

For any $\mathfrak{g}$-modules $M$ and $N$, the vector space
$\mathrm{Hom}_{\mathbb{C}}(M,N)$ has the natural structure of 
a $U(\mathfrak{g})$-$U(\mathfrak{g})$-bimodule. The subspace
$\cL(M,N)$  consisting of all
vectors of $\mathrm{Hom}_{\mathbb{C}}(M,N)$, 
the adjoint action of $\mathfrak{g}$ on which is locally
finite, is a $U(\mathfrak{g})$-$U(\mathfrak{g})$-subbimodule.
If $M=N$, then the image of $U(\mathfrak{g})$ in 
$\mathrm{Hom}_{\mathbb{C}}(M,M)$ belongs to 
$\cL(M,M)$ and this $U(\mathfrak{g})$-$U(\mathfrak{g})$-bimodule 
map is also an algebra map. The following is known as {\em Kostant's problem},
see \cite{Jo}:

\begin{problem}\label{prob4}
For which $w\in W$, the image of 
$U(\mathfrak{g})$ in $\mathrm{Hom}_{\mathbb{C}}(L_w,L_w)$
coincides with $\cL(L_w,L_w)$?
\end{problem}

This problem was studied, for example, in \cite{GJ,MS1,MS2,Ma1,Ma2,KhM,Kh},
where several partial results were obtained. However, the general case is very much open.
As already mentioned, we write $\mathbf{K}(w)$ for the statement 
``the image of  $U(\mathfrak{g})$ in $\mathrm{Hom}_{\mathbb{C}}(L_w,L_w)$
coincides with $\cL(L_w,L_w)$''. 

\subsection{K{\aa}hrstr{\"o}m's conditions}\label{s2.4}

For $y\in W$, we write $\mathbf{Kh}(y)$ for the statement 
``for all $x\neq z\in W$ such that
$\theta_x L_y\neq 0$ and $\theta_zL_y\neq 0$, we have 
$\theta_x L_y\not\cong \theta_zL_y$'' and write $[\mathbf{Kh}](y)$ for the statement 
``for all $x\neq z\in W$ such that
$\theta_x L_y\neq 0$ and $\theta_zL_y\neq 0$, we have 
$[\theta_x L_y]\neq [\theta_zL_y]$ in the Grothendieck group $\mathrm{Gr}(\cO^\mathbb{Z}_0)$''.

In particular, the equivalence of 
Conjecture~\ref{kconj-intro}\ref{kconj-1-intro} and  
Conjecture~\ref{kconj-intro}\ref{kconj-2-intro} can be
expresses as $\mathbf{K}(d)\Leftrightarrow\mathbf{Kh}(d)$, and 
the equivalence of Conjecture~\ref{kconj-intro}\ref{kconj-2-intro} and  
Conjecture~\ref{kconj-intro}\ref{kconj-3-intro} can be
expresses as $\mathbf{Kh}(d)\Leftrightarrow[\mathbf{Kh}](d)$,
for a Duflo element $d\in W$.

\section{Kazhdan-Lusztig combinatorics}

The Kazhdan-Lusztig conjecture from \cite{KL}, proved in \cite{BB,BK} (see also \cite{EW}), 
tells us that a large amount of information on $\cO^\mathbb{Z}_0$ is 
encoded in the Hecke algebra of $W$ and thus can be computed combinatorially.
We recall in this section some well-known constructions and facts around the Hecke algebra and their relation to $\cO^\mathbb{Z}_0$.

\subsection{Hecke algebra and Kazhdan-Lusztig basis}

Let $S\subset W$ be the set of simple reflections. Then $(W,S)$ is a Coxeter group. 
The Hecke algebra $H(W,S)$ associated to $(W,S)$ is a $\mathbb Z[v,v\inv]$-algebra 
generated by $H_s$, for $s\in S$, which satisfy the (Coxeter) braid relations and
the quadratic relation
\[(H_s+v)(H_s-v\inv)=0,\]
for all $s\in S$.
Given a reduced expression $w=st\cdots u$ of $w\in W$, we let $H_w=H_sH_t\cdots H_u$. 
The element $H_w$ is, in fact, independent of the choice of the reduced expression, 
and $\{H_w\}_{w\in W}$ is a ($\mathbb Z[v,v\inv]$-)basis of $H(W,S)$ called the {\em standard basis}.
Now consider the ($\mathbb Z$-algebra-)involution 
\[\overline{\phantom{A}} :H(W,S)\to H(W,S)\] 
determined by $\overline{v}= v\inv$ and $\overline{H_s} = H_s\inv$.
Then there is a unique element $\kl_w$ in $H(W,S)$ such that $\overline{\kl_w} = \kl_w$ and 
\[\kl_w = H_w + \sum_y p_{y,w}H_y,\]
for some $p_{y,w}\in v\mathbb Z[v]$.
The elements $\kl_w$, where $w\in W$, form a basis of $H(W,S)$ called the {\em Kazhdan-Lusztig (KL) basis}.
We refer to \cite{KL}, or \cite[\S 3-5]{luneq} for details (if referring to \cite{luneq}, note that we are in the special case $L(s)=1$ for all $s\in S$ and 
that our $v$ is denoted by $v\inv$, $H_w$ is denoted by $T_w$, and our $\kl_w$ is denoted by $c_w$ in \cite{luneq}).

We denote by $\boldsymbol{\mu}$ the Kazhdan-Lusztig $\mu$-function.

\subsection{Kazhdan-Lusztig theory}\label{ss:KLT}

Let $\cP$ be the monoidal category of (graded) projective functors on $\cO_0^\mathbb{Z}$ where 
indecomposables $\theta_w$, for $w\in W$, are normalized such that $\theta_w P_e \cong P_w$
holds in $\cO^\mathbb{Z}_0$. Then, for the split Grothendieck ring 
$\mathrm{Gr}_{\oplus}(\cP)$, we have
\begin{equation}
    \begin{split}
        \mathrm{Gr}_{\oplus}(\cP)^{\mathrm{op}}&\xrightarrow{\simeq} H(W,S)\\
        [\theta_w]&\mapsto \kl_w\\
    \end{split}
\end{equation}
as $\mathbb Z[v,v\inv]$-algebras so that $\langle 1\rangle$ on the left corresponds to $v$ on the right.
For example, $[\theta_w\langle m\rangle]\mapsto v^m\kl_w$. Furthermore,
\begin{equation}\label{eqn4}
    \begin{split}
        \mathrm{Gr}(  \cO_0^\Z)&\xrightarrow{\simeq} H(W,S)\\
        [P_w]&\mapsto \kl_w
    \end{split}
\end{equation}
as (right) modules over $\mathrm{Gr}_{\oplus}(\cP)^{\mathrm{op}}\cong H(W,S)$ 
so that $\langle 1\rangle$ on the left hand side corresponds to $v$ on the right hand side.
We note that, for $w\in W$, we have $[\Delta_w]\mapsto H_w$, where $\Delta_w$ denotes the (graded)
Verma module  with simple top $L_w$.

Using \eqref{eqn4}, various multiplicities in $\mathcal{O}_0^\mathbb{Z}$ can be described
in terms of the Kazhdan-Lusztig polynomials from \cite{KL}. In particular, the
graded composition multiplicities in $P_e$ are given by the 
corresponding coefficients of the Kazhdan-Lusztig polynomials.

\subsection{Kazhdan-Lusztig cells and the $\mathbf{a}$-function}\label{ss:cella}

For $x,y\in W$, write
\begin{equation}\label{hxyz}
    \kl_x\kl_y= \sum_{z} h_{x,y,z}\kl_z ,
\end{equation}
where $h_{x,y,z}\in \Z [v^{\pm 1}]$. 
In fact, we have $h_{x,y,z}\in \Z_{\geq 0} [v^{\pm 1}]\cap \Z[v+v\inv]$ by \cite{KL}.

Given $x,y\in W$, we say $x\leq_\mathtt{L} y$ if there exists $z\in W$ such that $\kl_y$ 
appears with a nonzero coefficient in $\kl_z\kl_x$, that is, $h_{z,x,y}\neq 0$. 
This defines the equivalence relation $\sim_\mathtt{L}$ on $W$.
A (Kazhdan-Lusztig) left cell is an equivalence class for the relation $\sim_\mathtt{L}$.
Similarly, define the right preorder $\leq_\mathtt{R}$
and (Kazhdan-Lusztig) right cells using multiplication on the right.
Finally, we define the two-sided preorder $\leq_\mathtt{J}$
and (Kazhdan-Lusztig) two-sided cells using multiplication on both sides.
The equivalence relation $\sim_\mathtt{J}$ is the minimum equivalence 
relation containing both $\sim_\mathtt{L}$ and $\sim_\mathtt{R}$.
We define the equivalence relation $\sim_\mathtt{H}$ as the
intersection of $\sim_\mathtt{L}$ and $\sim_\mathtt{R}$.
Equivalence classes for $\sim_\mathtt{H}$ are called  $\mathtt{H}$-cells.
A two-sided cell is called {\em strongly regular} if the intersection
of each left and each right cell inside this two-sided cell is
a singleton. A left (right) cell is {\em strongly regular} if it belongs
to a strongly regular two-sided cell.

We note that the left and right orders are the opposite of that in \cite{luneq}.
Our conventions are consistent with the previous papers \cite{Ma0,Ma2,CoM0,CoM} of the series.

Lusztig's a-function $\mathbf{a}:W\to \Z_{\geq 0}$ is defined as follows:
\begin{equation}
 \mathbf{a}(z):=\operatorname{max}_{x,y\in W}\{\operatorname{deg}h_{x,y,z}\}.
\end{equation}
The following facts can be found in \cite{luneq} (note that the conjectures P1-15 
in \cite[\S 13]{luneq} are proved in \cite[\S 14-15]{luneq} in our setting). 
Let $w_0$ denote the longest 
element in $W$.

\begin{proposition}\label{cella}
Let $x,y\in W$.
\begin{enumerate}
    \item \label{cella:i} $x\leq_\mathtt{L} y\iff x\inv\leq_\mathtt{R} y\inv$. 
    \item \label{cella:ii} $x\leq_\mathtt{L} y \iff w_0x\geq_\mathtt{L} w_0y$. 
    Furthermore, $x\leq_\mathtt{R} y \iff xw_0\geq_\mathtt{R} yw_0$.
    \item \label{cella:p4} Let $\mathtt{X}\in\{\mathtt{L},\mathtt{R},\mathtt{J}\}$. If $x\leq_\mathtt{X} y$, then $\mathbf{a}(x)\leq\mathbf{a}(y)$. 
    If $x<_\mathtt{X} y$, then $\mathbf{a}(x)<\mathbf{a}(y)$. 
    In particular, the $\mathbf{a}$-function is $\mathtt{J}$-cell invariant. 
    \item \label{cella:p9} If $\mathbf{a}(x)=\mathbf{a}(y)$, then $x\leq_\mathtt{L} y \Rightarrow x\sim_\mathtt{L} y$. Furthermore, if $\mathbf{a}(x)=\mathbf{a}(y)$, then $x\leq_\mathtt{R} y \Rightarrow x\sim_\mathtt{R} y$.
\end{enumerate}
\end{proposition}

Recall from
\cite{Lu1,Lu2}, that the $\mathbf{a}$-function can be also defined as the maximal (or minimal, depending on the normalization) 
possible degree of a Kazhdan-Lusztig polynomial between the
identity and an element of a given
left (or right) cell. As a consequence of this definition, for
any $w\in W$, we have
\begin{equation}\label{eq9}
[P_e:L_w\langle i\rangle]\neq 0\quad\text{ implies }\quad
-\mathbf{a}(w)\geq i\geq -\ell(w),
\end{equation}
where $\ell(w)$ denotes the length of $w$,
moreover, the inequality $-\mathbf{a}(w)\geq i$ is strict
unless $w$ is a Duflo element and in the latter case
\begin{equation}\label{eq9-1}
[P_e:L_w\langle -\mathbf{a}(w)\rangle]=1.
\end{equation}

For each right cell $\mathcal{R}$ in $W$, we have the Serre subcategory 
$\mathcal{O}_{0}^{\hat{\mathcal R}}$ of $\mathcal{O}_{0}$ whose simples are
$L_w$, for all $w\in W$ such that $w\leq_{\mathtt{R}}\mathcal{R}$,
see \cite{MS1}.

\subsection{Asymptotic rings}\label{ss:asym}

We introduce the asymptotic ring $A(W)=A(W,S)$ defined by Lusztig (and called ``the ring $J$'' in \cite{luneq}). It has a ($\mathbb Z$-)basis $\{t_w\}_{w\in W}$ whose multiplication is defined as
\begin{equation}\label{tt}
  t_xt_y=\sum_{z\in W}\gamma_{x,y,z\inv}t_z,  
\end{equation}
where $\gamma_{y,x,z\inv}\in \mathbb Z_{\geq 0}$ is the coefficient of $\theta_z\langle \mathbf{a}(z)\rangle$ in the decomposition of $\theta_x\theta_y$, i.e., the top degree coefficient in $h_{y,x,z}$ (see \eqref{hxyz}).
One can check that $\gamma_{x,y,z\inv}=0$ unless $y$ and $x\inv$ belong to the same right cell. 
The basis elements $t_d$ corresponding to Duflo elements $d\in W$ are local identities in $A(W)$
in the following sense:

\begin{lemma}\label{[d unit]}
Let $d\in W$ be a Duflo element in $\mathcal{H}=\mathcal{L}\cap \mathcal{L}\inv$. Here, $\mathcal{L}$ is the left cell containing $d$ and $\mathcal{L}\inv$ is the right cell containing $d\inv=d$.
We have
\begin{equation}
    t_dt_x = t_x\quad\text{ and } \quad t_yt_d = t_y.
\end{equation}
for each $y\in \mathcal{L}$ and $x\in \mathcal{L}\inv$ 
and
\begin{equation}\label{t't}
    t_{x\inv}t_x = t_d + \sum_{d\neq z\in \mathcal L\cap \mathcal{L}\inv}\gamma_{x\inv,x,z\inv}t_z.
\end{equation}
\end{lemma}

\begin{proof}
This follows from P5, P7, P8, P13 in \cite[\S14.1]{luneq} and positivity of $h_{x,y,z}$. 
\end{proof}

\subsection{Graded composition multiplicities in $\cO^\mathbb{Z}_0$}

By \S\ref{ss:KLT} (and \cite{KL,BB,BK}), composition multiplicities of many important objects in $\cO^{\mathbb Z}_0$ 
can be computed purely inside the Hecke algebra $H(W,S)$.

\begin{proposition}
For $x,y,w\in W$, we have 
\begin{equation} \label{gdim=h}
[\theta_x L_y:L_{z}]=h_{z,x^{-1},y},
\end{equation}
where $[\theta_x L_y:L_{z}]$ denotes the graded composition multiplicity viewed as an
element in $\mathbb{Z}[v,v^{-1}]$. In particular, for $i\in\mathbb{Z}$, we have
\begin{equation}\label{gdim=h.2}
[\theta_x L_y : L_z\langle i\rangle]=[\theta_z L_{y^{-1}} : L_x\langle i\rangle].
\end{equation}
\end{proposition}

\begin{proof}
Formula \eqref{gdim=h} follows, by adjunction, from the observation that the
multiplicity of $L_{z}\langle i\rangle$ in $\theta_x L_y$ equals
\begin{displaymath}
\dim\mathrm{hom}(\theta_{z}P_e\langle i\rangle,\theta_x L_y)=
\dim\mathrm{hom}(\theta_{x\inv}\theta_{z}P_e\langle i\rangle, L_y).
\end{displaymath}
By a similar adjunction, \eqref{gdim=h.2} is equivalent to
\begin{displaymath}
\mathrm{hom}(\theta_{x^{-1}}\theta_z L_{y^{-1}},I_e\langle i\rangle)\cong
\mathrm{hom}(\theta_{z^{-1}}\theta_x L_y,I_e\langle i\rangle),
\end{displaymath}
where $I_e$ is the indecomposable injective envelope of $L_e$ with socle
concentrated in degree zero.
Also, by adjunction, for $a,b\in W$ and $j\in\mathbb{Z}$, we have
\begin{displaymath}
\dim \mathrm{hom}(\theta_{a^{-1}} L_b,I_e\langle j\rangle) = \delta_{a,b}\delta_{j,0}.
\end{displaymath}
As $\theta_{x^{-1}}\theta_z$ is adjoint to 
$\theta_{z^{-1}}\theta_x$ and $\theta_y$ is adjoint to $\theta_{y^{-1}}$, 
the multiplicity of $\theta_y$ in $\theta_{x^{-1}}\theta_z$ coincides
with the multiplicity of  $\theta_{y^{-1}}$ in $\theta_{z^{-1}}\theta_x$.
The claim follows.
\end{proof}

\section{Proof of Theorem~\ref{thm-kmm-original}}\label{s7}

For $w\in W$, we denote by $T_w$ the indecomposable tilting module in 
$\mathcal{O}_0$ with highest weight $w\cdot 0$. 

We record a lemma which follows directly from, for example, \cite[Lemma~13(a)]{MM1}.
\begin{lemma}\label{nlem1}
Let $\theta$ be a  projective functor on $\mathcal O_0$ and $M\in\mathcal O_0$.
If $y\in W$ is such that $[\theta M:L_y]\neq 0$, then there exists $x\in W$ such that $[M:L_x]\neq 0$ and $y\leq_\mathtt{R} x$.
\end{lemma}

We also need the following statement.

\begin{lemma}\label{lem7.2-1}
Assume $M\in\mathcal{O}_0$. Let $i\in\mathbb{Z}$ and $w\in W$ 
be such that  $\mathrm{Ext}_{\mathcal{O}}^i(M,T_w)\neq 0$. Then we have
\begin{equation}\label{eq7.2-2}
i\geq \mathrm{min}\{\mathbf{a}(x^{-1}w_{0})\,:\, x\in W\text{ such that }[M:L_x]\neq 0\}. 
\end{equation}
\end{lemma}

\begin{proof}
As $T_w=\theta_{w_{0}w}T_{w_0}$, by adjunction, we have 
\begin{displaymath}
\mathrm{Ext}_{\mathcal{O}}^i(M,T_w)\cong 
\mathrm{Ext}_{\mathcal{O}}^i(\theta_{w^{-1}w_0}M,T_{w_0}).
\end{displaymath}
By Lemma \ref{nlem1}, any $y\in W$ such that $[\theta_{w^{-1}w_0}M:L_y]\neq 0$ satisfies
$y\leq_\mathtt{R} x$, for some $x\in W$ such that $[M:L_x]\neq 0$.
For such $x$ and $y$, we have $w_0y\geq_\mathtt{R} w_0x$ and thus 
$\mathbf{a}(w_0y)\geq \mathbf{a}(w_0x)$ (see Proposition \ref{cella}). In particular,
\begin{gather*}
\mathrm{min}\{\mathbf{a}(x^{-1}w_{0})\,:\, x\in W\text{ such that }[M:L_x]\neq 0\}=\\
\mathrm{min}\{\mathbf{a}(x^{-1}w_{0})\,:\, x\in W\text{ such that }[\theta_{w^{-1}w_0}M:L_x]\neq 0\}. 
\end{gather*}
Therefore, it is enough to prove the claim for $w=w_0$, in which case $T_w=T_{w_0}=L_{w_0}$.

Let $\mathcal{I}^{\bullet}_{w_0}$ be the minimal injective resolution of $L_{w_0}$. 
The assumption implies that there is a nonzero morphism from $M$ to $\mathcal I^i_{w_0}$, i.e., there is $x\in W$ such that $[M:L_x]\neq 0$ and $I_x\subseteq \mathcal I^i_{w_0}$.
Recall from \cite{So} that $\mathcal{O}_0$ is Koszul self-dual
and this self-duality maps 
$L_w$ to $P_{w^{-1}w_0}$ and $I_w$ to $L_{w^{-1}w_0}$.
Consequently, the minimal injective resolution $\mathcal{I}^{\bullet}_{w_0}$
of $L_{w_0}$ corresponds to the dominant projective module $P_e$
in the Koszul dual picture. 
Thus, for $x\in W$ as above, $I_x$ appearing as a summand of 
$\mathcal{I}^{i}_{w_0}$ implies that $L_{x^{-1}w_0}$ appears as a composition 
subquotient of $P_e$ in degree $i$. 
By \eqref{eq9}, we have $\mathbf{a}(x^{-1}w_0)\leq i\leq \ell(x^{-1}w_0)$, for such $x$.
This implies the inequality in \eqref{eq7.2-2} and completes the proof.
\end{proof}

\begin{proof}[Proof of Theorem~\ref{thm-kmm-original}]
Let $x,y,z\in W$.
Koszul-Ringel self-duality of $\{\theta_aL_b\,:\,a,b\in W\}$ from \cite[Theorem~16]{Ma2} 
maps a non-zero homomorphism from $\theta_xL_y$ to $L_z\langle i\rangle$ to a non-zero element
in 
\begin{displaymath}
\mathrm{Ext}^i_{\mathcal{O}}(\theta_{y^{-1}w_0}L_{w_0x^{-1}},T_{w_0z^{-1}w_0}).
\end{displaymath}
By Lemma~\ref{nlem1}, $\theta_{y^{-1}w_0}L_{w_0x^{-1}}$ consists of composition factors isomorphic to $L_z$ with $z\leq_\mathtt{R} w_0x^{-1}$. Equivalently, we have $z^{-1}\leq_\mathtt{L} xw_0$ yielding $x\leq_\mathtt{L} z^{-1}w_0$ by 
Proposition~\ref{cella}\eqref{cella:ii}.
For all such $z$, we have $\mathbf a(z^{-1}w_0)\geq \mathbf a(x)$ by Proposition~\ref{cella}\eqref{cella:p4}.
Now the claim of Theorem~\ref{thm-kmm-original} follows 
from Lemma~\ref{lem7.2-1}.
\end{proof}

\section{$\mathbf{KMM}$ vs $\mathbf{KM}$}\label{s4}

In this section we establish a connection between $\mathbf{KMM}$
and $\mathbf{KM}$.

\subsection{The graded endomorphism algebra of $\theta_x L_y$}\label{s4.1}

For $x,y\in W$, we consider the module $\theta_x L_y\in \mathcal{O}_0^{\mathbb{Z}}$ and
its endomorphism algebra $\mathrm{End}(\theta_x L_y)$ which is naturally $\mathbb{Z}$-graded.

\begin{lemma}\label{lem11}
The natural  $\mathbb{Z}$-grading on $\mathrm{End}(\theta_x L_y)$
is non-negative in the sense that all components with negative degrees are zero.
\end{lemma}

\begin{proof}
By \cite[Theorem~16]{Ma2}, the module $\theta_x L_y$ is Koszul-Ringel dual to
$\theta_{y^{-1}w_0} L_{w_0x^{-1}}$, where $w_0$ is the longest element of $W$.
Under this duality, the endomorphism algebra of $\theta_x L_y$ is 
mapped to the algebra of diagonal self-extensions for $\theta_{y^{-1}w_0} L_{w_0x^{-1}}$.
The grading of the latter algebra is manifestly non-negative,
which implies the claim.
\end{proof}

\begin{corollary}\label{cor12}
If $\dim \mathrm{End}(\theta_x L_y)_0=1$, then $\mathbf{KM}(x,y)$ is true.
\end{corollary}

\begin{proof}
Every idempotent of the non-negatively graded algebra $\mathrm{End}(\theta_x L_y)$
must be homogeneous of degree zero. Therefore the assumption
$\dim \mathrm{End}(\theta_x L_y)_0=1$
implies that $\mathrm{End}(\theta_x L_y)_0\cong\mathbb{C}$, that is, the only idempotents
of $\mathrm{End}(\theta_x L_y)$ are $0$ and $1$. This means that 
$\mathrm{End}(\theta_x L_y)$ is local and hence $\theta_x L_y$ is indecomposable.
\end{proof}

\subsection{$\mathbf{KMM}(*_d,y)$ implies $\mathbf{KM}(*,y)$
in type $A$}\label{s4.3}

Assume now that $\mathfrak{g}$ is of type $A$. Then the cell structure of
$W$ is especially nice. In particular, if $\mathcal{L}$ and $\mathcal{R}$
are a left and a right cell inside the same two-sided cell, then
$|\mathcal{L}\cap \mathcal{R}|=1$. Moreover, if $\mathcal{R}=\{w^{-1}\,:\, w\in \mathcal{L}\}$,
then $\mathcal{L}\cap \mathcal{R}=\{d\}$, where $d$ is a Duflo element.

\begin{lemma}\label{lem10}
For $x\in W$ in type A and $d$ a Duflo element such that $x\sim_R d$, we have
\begin{displaymath}
\theta_{x^{-1}}\theta_x\cong \theta_d\langle-\mathbf{a}(d)\rangle
\oplus \bigoplus_{w\in W}\bigoplus_{i>-\mathbf{a}(w)} \theta_w\langle i\rangle^{\oplus m_{w,i}}.
\end{displaymath}
\end{lemma}
\begin{proof}
Since $\mathcal{R}\cap \mathcal{R}\inv=\{d\}$, where $x\in \mathcal R$, the claim follows from \eqref{t't} in Lemma \ref{[d unit]}.
\end{proof}

\begin{proposition}\label{prop14}
Assume that $\mathfrak{g}$ is of type $A$ and  $x,y\in W$.  
Then $\mathbf{KMM}(*_d,y)$  implies $\mathbf{KM}(x,y)$.
\end{proposition}

\begin{proof}
By adjunction, we have
\begin{displaymath}
\mathrm{hom}(\theta_{x} L_y,\theta_{x} L_y)\cong
\mathrm{hom}(\theta_{x^{-1}}\theta_{x} L_y, L_y).
\end{displaymath}
Given Theorem~\ref{thm-kmm-original}, from Lemma~\ref{lem10} it follows that
the only term in the decomposition of $\theta_{x^{-1}}\theta_{x}$
which can contribute to a non-zero element of 
$\mathrm{hom}(\theta_{x^{-1}}\theta_{x} L_y, L_y)$
is $\theta_{d'}\langle-\mathbf{a}(x)\rangle$, where $d'\sim_\mathtt{R} x$
is a Duflo element. Therefore, we just need to show that
\begin{displaymath}
\dim\mathrm{hom}(\theta_{d'} L_y, L_y\langle\mathbf{a}(x)\rangle)=1.
\end{displaymath}
If $\theta_{d'} L_y\neq 0$, this is guaranteed by $\mathbf{KMM}(d',y)$.
\end{proof}

\section{Extensions between translated simple modules}\label{s8}

\subsection{Monotonicity of projective dimension}\label{s8.2}

One of our main observations in this section is
the following statement in the spirit of \cite{Ma2}
and \cite{CoM}. 

\begin{theorem}\label{cong-kmm-2}
Let $x,x',y\in W$ be such that $x>_\mathtt{R} x'$ and $\theta_{x'}L_y\neq 0$.
Then 
\begin{displaymath}
\mathrm{proj.dim}(\theta_xL_y)<\mathrm{proj.dim}(\theta_{x'}L_y). 
\end{displaymath}
\end{theorem}

Note that, in the setup of Theorem~\ref{cong-kmm-2}, 
the inequality $\mathrm{proj.dim}(\theta_xL_y)\leq\mathrm{proj.dim}(\theta_{x'}L_y)$
follows easily applying projective functors to a minimal projective resolution
of $\theta_{x'}L_y$, where $\theta_xL_y$ can be eventually found as a direct summand
of the homology in the homological position zero,
since $x>_\mathtt{R} x'$.

We note that the main result of \cite{Ma2} provides an explicit formula
for $\mathrm{proj.dim}(\theta_yL_{w_0})$ in terms of the $\mathbf{a}$-function
and thus implies Theorem~\ref{cong-kmm-2} in the case $y=w_0$ by since the
$\mathbf{a}$-function is strictly monotone along any of the Kazhdan-Lusztig orders (see Proposition \ref{cella} \eqref{cella:ii}). More generally, in the case $y=w_0^{\mathfrak{p}}w_0$,
where $\mathfrak{p}$ is a parabolic subalgebra of $\mathfrak{g}$,
Theorem~\ref{cong-kmm-2} follows from \cite[Table~2 and Theorem~4.1(i)]{CoM}.

Let $M$ be a $\mathbb{Z}$-graded module. Then the {\em graded length}
of $M$ is defined as the difference between the maximal and the minimal
degrees of non-zero components of $M$. If $M=0$, then the graded length
is, by convention, $-\infty$. For modules concentrated in a single
degree, the graded length is $0$. For example, the projective dimension
of $M$ is the same as the graded length of a minimal projective
resolution of $M$ consider as a module in an appropriate graded
category with a $\mathbb{Z}$-grading given by the homological degree.

For $x,y\in W$, denote by $g_{x,y}$ the graded length of $\theta_xL_y$, 
viewed as an object of the category of linear complexes of tilting modules in $\mathcal{O}$.
From \cite[Theorem~16]{Ma2}, it follows that $g_{x,y}$ is always even.
Applying projective functors, 
it is easy to see that
\begin{equation}\label{eqsec8-3}
g_{x,y}\leq g_{x',y}\quad \text{ if }\quad x\geq_\mathtt{R} x'. 
\end{equation}

\begin{proposition}\label{prop8.5-1}
Consider $x,y\in W$ such that $\theta_xL_y\neq 0$ and 
let $k:=\mathrm{proj.dim}(\theta_xL_y)$. Then we have:
\begin{enumerate}
\item \label{prop8.5-1.1} $k=\mathbf{a}(w_0x)+\half g_{x,y}=:k'$.
\item \label{prop8.5-1.2} For any $z\in W$ such that
$z\not\sim_\mathtt{L} x$, we have $\mathrm{Ext}^k(\theta_xL_y,L_z)=0$.
\end{enumerate}
\end{proposition}

\begin{proof}
Let $z\in W$ be such that $\mathrm{Ext}^k(\theta_xL_y,L_z)\neq 0$.
Let $d$ be the Duflo element in the left cell of $x$. Then
$\theta_d\theta_x=\theta_x\oplus \theta$, for some projective functor
$\theta$, in particular, $\mathrm{Ext}^k(\theta_d\theta_xL_y,L_z)\neq 0$.
By adjunction, we obtain $\mathrm{Ext}^k(\theta_xL_y,\theta_dL_z)\neq 0$.
We use the Koszul-Ringel duality of \cite[Theorem~16]{Ma2} and represent both 
$\theta_xL_y$ and $\theta_dL_z$ as complexes of tilting modules.
We call these complexes $\mathcal{X}^{\bullet}$ and $\mathcal{Y}^{\bullet}$,
respectively. The minimal non-zero position of $\mathcal{X}^{\bullet}$
is $-\frac{1}{2}g_{x,y}$. The maximal non-zero position of
$\mathcal{Y}^{\bullet}$ is $\frac{1}{2}g_{d,z}$.
As $d\sim_\mathtt{J} x$, from \cite[Proposition~1]{Ma2} we get the inequality
$\frac{1}{2}g_{d,z}\leq\mathbf{a}(w_0d)= \mathbf{a}(w_0x\inv)=\mathbf{a}(w_0x)$, moreover,
$\frac{1}{2}g_{d,z}=\mathbf{a}(w_0x)$ if and only if $w_0z\inv\sim_\mathtt{L}w_0d$. 
Note that $w_0z\sim_\mathtt{L}w_0d$ is equivalent
to $z\sim_\mathtt{L} d$ in which case we have $z\sim_\mathtt{L}x$.
This implies that $k\leq k'$ and that $k=k'$ holds exactly when $z\sim_\mathtt{L} x$. 

Therefore, we just need to find some $z\in W$ such that $z\sim_\mathtt{L} x$ and
\begin{displaymath}
\mathrm{Ext}^{k'}(\theta_xL_y,L_z)\neq 0. 
\end{displaymath}
Let $L_u$ be some simple module which appears in the graded degree
$-\frac{1}{2}g_{x,y}$, which is the extreme degree, of the module $\theta_{y^{-1}w_0}L_{w_0x^{-1}}$.
By Lemma \ref{nlem1} $u$ belongs to the right cell $\mathcal{R}$ of $w_0x^{-1}$.
Since all indecomposable projective-injective modules in 
$\mathcal{O}^{\hat{\mathcal{R}}}$ are of the form  $\theta_{z^{-1}w_0}L_{w_0d}$ for $z\sim_\mathtt{L}x$ (see \cite{Ma2}),
there exists $z\sim_{\mathtt{L}} x$ such that $\theta_{z^{-1}w_0}L_{w_0d}$ is the injective envelope of
$L_u$ in $\cO^{\hat{\mathcal{R}}}$. For this $z$, the factor $L_u$ appears in the graded degree
$\frac{1}{2}g_{d,z}$ of the module $\theta_{z^{-1}w_0}L_{w_0d}$.
Since any isomorphism between the corresponding summands of
$\mathcal{X}^{-\frac{1}{2}g_{x,y}}$
and $\mathcal{Y}^{\frac{1}{2}g_{d,z}}$
gives rise a non-zero extension of degree $k'$
between $\theta_xL_y$ and $\theta_dL_z$ (e.g., by the argument from \cite[Theorem~1]{MO},
a), the claim follows.
\end{proof}

\begin{proof}[Proof of Theorem~\ref{cong-kmm-2}.]
By Proposition~\ref{prop8.5-1}\eqref{prop8.5-1.1} we need to prove that $\mathbf{a}(w_0x)+\half g_{x,y}$ is strictly monotone along
the right Kazhdan-Lusztig order with respect to $x$. The term  $\mathbf{a}(w_0x)$ is 
strictly monotone by Proposition \ref{cella} \eqref{cella:ii},\eqref{cella:p4}.
The claim follows by \eqref{eqsec8-3}.
\end{proof}

Directly from Proposition~\ref{prop8.5-1}\eqref{prop8.5-1.2}, we obtain:

\begin{corollary}\label{cor:kconj-301}
Conjecture~\ref{kconj2-intro} is true.
\end{corollary}

We record a special case of Proposition \ref{prop8.5-1}\eqref{prop8.5-1.1}.

\begin{corollary}\label{prop8.4-1}
Let $x,y\in W$ be such that $\theta_x L_y\neq 0$ and $x\sim_\mathtt{J} y$. 
Then
$$\mathrm{proj.dim}(\theta_x L_y)=2\mathbf{a}(w_0x)=\mathbf{a}(w_0x)+\mathbf{a}(y\inv w_0).$$
\end{corollary}
\begin{proof}
This follows from Proposition \ref{prop8.5-1} \eqref{prop8.5-1.1} and \cite[Proposition 1]{Ma2}.
\end{proof}

\subsection{Theorem~\ref{cong-kmm-2} does not naively extend to
the Koszul-Ringel dual}\label{s8.3}

An alternative description of $g_{x,y}$ is given by:

\begin{proposition}\label{prop82}
For $x,y\in W$ such that $\theta_xL_y\neq 0$, we have
\begin{displaymath}
g_{x,y}=\mathrm{max}\{i\,:\,\mathrm{Ext}^i(\theta_xL_y,\theta_xL_y)\neq 0\}. 
\end{displaymath}
\end{proposition}

\begin{proof}
By representing $\theta_xL_y$ as a complex $\mathcal{X}^{\bullet}$ of tilting modules,
we can compute $\mathrm{Ext}^i(\theta_xL_y,\theta_xL_y)$ via homomorphisms from
$\mathcal{X}^{\bullet}$ to $\mathcal{X}^{\bullet}[i]$ in the homotopy category of 
titling modules. If $i>g_{x,y}$, then the homomorphism space is obviously zero
as non-zero components of $\mathcal{X}^{\bullet}$ and $\mathcal{X}^{\bullet}[i]$
never match. 

Because of the self-duality of $\theta_xL_y$ combined with the
Koszul-Ringel self-duality of $\{\theta_uL_v\,:\,u,v\in W\}$ from
\cite[Theorem~16]{Ma2}, the maximal and the minimal non-zero components
of $\mathcal{X}^{\bullet}$ are isomorphic as objects in $\mathcal{O}$.
Therefore, when $i=g_{x,y}$, the argument from the proof of \cite[Theorem~1]{MO} shows 
that an isomorphism from the minimal component in $\mathcal{X}^{\bullet}$ 
to the maximal component in $\mathcal{X}^{\bullet}[i]$ induces a non-zero
homomorphism between complexes in the 
homotopy category of complexes of tilting modules. The claim follows.
\end{proof}

In contrast to Theorem~\ref{cong-kmm-2}, the function $x\mapsto g_{x,y}$,
for $y$ fixed, does not have to be strictly monotone with respect to  the
right order on $W$. Indeed, if $u\in W$ is the longest element of some 
parabolic subgroup, then the graded length of each non-zero $\theta_uL_v$, where $v\in W$,
equals $2\ell(u)$. By Koszul-Ringel duality, this gives, in the case $y=w_0u\inv$,  
that $g_{x,y}=2\ell(u)$, for any $x$ for which $g_{x,y}\neq-\infty$.
In other words, in this case, the function $x\mapsto g_{x,y}$ is, in fact,
constant. This example and the discussion above 
imply the following corollary:

\begin{corollary}\label{cor84}
Let $\mathfrak{p}$ be a parabolic subalgebra of $\mathfrak{g}$ and
$T$ any non-zero tilting module in the parabolic category $\mathcal{O}_0^{\mathfrak{p}}$,
then
\begin{displaymath}
\mathrm{max}\{i\,:\,\mathrm{Ext}^i_{\mathcal{O}}(T,T)\neq 0\}=2\ell(w_0^{\mathfrak{p}}). 
\end{displaymath}
\end{corollary}

\begin{proof}
This follows from the above discussion noting that all indecomposable tilting modules in 
$\mathcal{O}_0^{\mathfrak{p}}$ are of the form $\theta_xL_{w_0^{\mathfrak{p}'}w_0}$,
for some $x$, where $\mathfrak{p}'$ is the parabolic subalgebra obtained from
$\mathfrak{p}$ using $w_0$.
\end{proof}

We emphasize that the extensions in the formulation of Corollary~\ref{cor84}
are taken in $\mathcal{O}$ and not in $\mathcal{O}^{\mathfrak{p}}$. Indeed,
$T$ is assumed to be a tilting module in $\mathcal{O}^{\mathfrak{p}}$ and,
as such, is ext-self-orthogonal in $\mathcal{O}^{\mathfrak{p}}$.

\subsection{The $\mathbf{b}$-function}\label{ss:b}

We look further into the number $g_{x,y}$ appear in Proposition \ref{prop82}.
We introduce a function which refines the $\mathbf{a}$-function. 
For $x,y\in W$, consider the non-negative integer
\begin{equation}
    \mathbf b(x,y):= \operatorname{max}\{\operatorname{deg}h_{z,x^{-1},y}\,:\,z\in W\} 
\end{equation}
which we take as the definition of the function 
$\mathbf{b}:W\times W \to \mathbb N\sqcup \{-\infty\}$ 
(by our convention the degree of the zero polynomial is $-\infty$).

By \eqref{gdim=h},
$\mathbf b(x,y)$ gives the height (the highest degree of the composition factors) of the module $\theta_xL_y$. Since the latter is self-dual, we have
\begin{equation}\label{2b=gl}
 2\mathbf b(x,y) = \text{the graded length of }\theta_xL_y .
\end{equation}

Using the Koszul-Ringel duality from \cite{Ma2}, we also have
\begin{equation}\label{2b=td}
2\mathbf b(x,y) = \operatorname{tilt.dim}\theta_{y\inv w_0}L_{w_0x\inv}    
\end{equation}
where $\operatorname{tilt.dim}M$ denotes the length of a minimal complex of
tilting modules representing $M$. Consequently, we have
\begin{equation}\label{2b=g}
2\mathbf b(x,y) = g_{y\inv w_0,w_0x\inv}. 
\end{equation}

Thus we can rewrite Proposition~\ref{prop8.5-1}\eqref{prop8.5-1.1} as
\begin{equation}\label{pd=a+b}
\operatorname{proj.dim}(\theta_xL_y)= \mathbf{a}(w_0x)+\mathbf{b}(y\inv w_0,w_0x\inv).
\end{equation}
On the one hand, this provides a formula for the projective dimension
of $\theta_xL_y$  purely in
terms of Kazhdan-Lusztig combinatorics. On the other hand, this
gives another characterization of the $\mathbf b$-function.

\begin{proof}
This follows directly from \eqref{pd=a+b} and Proposition~\ref{propnn125}\eqref{propnn125.1}.
\end{proof}

\begin{proposition}\label{leftcell}
The value $\mathbf{b}(x,y)$ is $-\infty$ if and only if $x\inv\not\leq_\mathtt{L} y$.
Moreover, if $y\sim_\mathtt{L} y'$, then $\mathbf b(x,y)=\mathbf b(x,y')$.
\end{proposition}

\begin{proof}
The first statement follows from \eqref{2b=gl}. If $y\sim_\mathtt{L} y'$, then $y\inv w_0\sim_\mathtt{R} (y')\inv w_0$. The second statement follows from \eqref{2b=td}. 
Alternatively, given a tilting resolution $T_\bullet$ of $\theta_{y\inv w_0} L$, a tilting resolution of $\theta_{(y')\inv w_0}L$ is obtained as a direct summand of $\theta T_\bullet$ for some projective functor $\theta$. This implies $\mathbf b(x,y)\geq \mathbf b(x,y')$ and the desired equality holds by symmetry.
\end{proof}

\begin{proposition}\label{propnn125}
Let $x,y\in W$ be such that $x\inv\leq_\mathtt{L} y$.
\begin{enumerate}
    \item\label{propnn125.1} If $x\sim_\mathtt{J} y$ (i.e., $x\inv\sim_\mathtt{L} y$), then we have $\mathbf a(x) =\mathbf b(x,y)=\mathbf a(y)$;
    \item\label{propnn125.2} If $x\not\sim_\mathtt{J} y$ (i.e., $x\inv <_\mathtt{L} y$), then $
        \mathbf a(x) \leq \mathbf b(x,y)<\mathbf a(y);$
    \item\label{propnn125.3} $\mathbf b(x,y)\leq \ell(x)$;
    \item\label{propnn125.4} If $x=w_0^{\mathfrak{p}}$, for some parabolic
    subalgebra $\mathfrak{p}$, then 
    $\mathbf a(x)=\mathbf b(x,y)=\ell(x)$;
    \item\label{propnn125.5} If $y=w_0w_0^{\mathfrak{p}}$, 
    for some parabolic subalgebra $\mathfrak{p}$, then 
    $\mathbf b(x,y)=\mathbf s_{\lambda_{\mathfrak{p}^{w_0}}}(w_0x\inv)-\mathbf a(y)$, where 
    $\mathbf s_{\lambda_{\mathfrak{p}^{w_0}}}(w)$ denotes the projective dimension of $L(w.\lambda_{\mathfrak{p}^{w_0}})$ 
    for a singular weight $\lambda_{\mathfrak{p}^{w_0}}$  with stabilizer given by the $w_0$-conjugate of the
    parabolic subgroup of $W$ corresponding to $\mathfrak{p}$, cf. \cite{CoM}.
\end{enumerate}
\end{proposition}

\begin{proof}
By \eqref{2b=gl}, we have $\mathbf a (x)\leq \mathbf b(x,y)\leq \ell(x)$ and $\mathbf b(x,y)\leq \mathbf{a}(y)$ is clear from the definition. The latter is equality if and only if $x\sim_\mathtt{J} y$ by Proposition \ref{prop8.4-1}. This proves claims \eqref{propnn125.1},
\eqref{propnn125.2} and \eqref{propnn125.3}.
Claim \eqref{propnn125.4} follows from claims \eqref{propnn125.1},
\eqref{propnn125.2} and \eqref{propnn125.3} since 
$\mathbf a(w_0^{\mathfrak{p}})=\ell(w_0^{\mathfrak{p}})$.

Claim \eqref{propnn125.5} follows from Proposition \ref{prop8.5-1} and \eqref{2b=g}, 
after two observations. The first observation is that 
$L(w_0x\inv.\lambda)$ has a linear projective resolution 
in its singular  block $\cO_\lambda$ and this resolution is mapped to a linear projective resolution of $\theta_{w_0^{\mathfrak{p}}}L_{w_0x\inv}$ by translating out of 
the wall from $\cO_\lambda$ to $\cO_0$. The second observation is that
$\mathbf{a}(y)=\mathbf{a}(w_0y\inv w_0)$.
\end{proof}

\begin{remark}
In the case $y=w_0^{\mathfrak{p}}w_0$, we alternatively have 
$\theta_xL_y=\theta_xL_{w_0^{\mathfrak{p}}w_0}$ is the indecomposable 
tilting module with highest weight $w_0^{\mathfrak{p}}w_0x$ in the 
principal block $\mathcal{O}_0^{\mathfrak{p}}$ of the
parabolic category $\mathcal{O}^{\mathfrak{p}}$.
According to \cite[Table~2]{CoM}, the graded length of this 
parabolic tilting module is given by 
$2(\mathbf{s}_{\lambda}((w_0^{\mathfrak{p}}w_0x)\inv w_0^{\mathfrak{p}})-\mathbf{a}(w_0w_0^{\mathfrak{p}}))$. Therefore we have 
$\mathbf{b}(x,y)=\mathbf{s}_{\lambda}(w_0(w_0x\inv) w_0)-\mathbf{a}(w_0yw_0)$
which agrees with Proposition~\ref{propnn125}\eqref{propnn125.5}.
\end{remark}

Note that Proposition~\ref{propnn125}\eqref{propnn125.5} provides an explicit
description of the function $\mathbf s_{\lambda}$ in terms of Kazhdan-Lusztig combinatorics.

\section{$\mathbf{K}$ and $\mathbf{KMM}$}\label{s3}

\subsection{Category $\mathcal{O}$ via Harish-Chandra bimodules}\label{s3.1}

Denote by $\mathcal{H}$ the category of {\em Harish-Chandra}
$U(\mathfrak{g})$-$U(\mathfrak{g})$-bimodules, that is,
finitely generated $U(\mathfrak{g})$-$U(\mathfrak{g})$-bimodules
on which the adjoint action of $\mathfrak{g}$ is locally finite
and has finite multiplicities, see \cite[Kapitel~6]{Ja}.
Note that $\mathcal H$ is equipped with a tensor product given by
tensoring over $U=\ug$ (however, $\mathcal H$ is not monoidal as the regular
bimodule $\ug$ is not a Harish-Chandra bimodule).

Let $\mathbf{m}$ denote the maximal ideal of $Z(\mathfrak{g})$
which annihilates the trivial $\mathfrak{g}$-module. 
Let $I$ be the kernel of the
surjection from $Z(\mathfrak{g})$ onto $\mathrm{End}(P_{w_0})$,
see Endomorphismensatz in \cite{So}.
Let $\hciii$ denote the full subcategory of 
$\mathcal{H}$ consisting of all bimodules $X$ such that
$X\mathbf{m}^i=0$ and $\mathbf{m}^iX=0$, for $i\gg 0$. 
Similarly let $\hci$ be the full subcategory of 
$\mathcal{H}$ (or of $\hciii$) consisting of bimodules $X$ such that
$X\mathbf{m}=0$ and $\mathbf{m}^iX=0$ for $i\gg 0$.
Let $\hcii$ be the full subcategory of $\hciii$ consisting of bimodules $X$ such that $XI=0$ and $\mathbf{m}^iX=0$ for $i\gg 0$.
By \cite{BG}, we have mutually inverse equivalences of categories
as follows:
\begin{equation}\label{eq1}
\xymatrix{
\mathcal{O}_0\ar@/^/[rrr]^{\ccL(P_e,{}_-)}&&&
{}^{\infty}_{\,\,0}\mathcal{H}^1_0\ar@/^/[lll]^{{}_-\otimes_{U(\mathfrak{g})}P_e}
}
\end{equation}

Similarly,  we have a monoidal equivalence
\begin{equation}\label{hcii=P}
(\overline{\cP},\circ)\cong (\hcii,\otimes_U),
\end{equation}
where $\overline{\cP}$ denotes the projective abelianization of $\cP$,
see \cite[\S~3.5]{MM1}.

\subsection{Internal hom of projective functors}

The monoidal category $(\hcii,\otimes_U)$ acts (in the sense of \cite[Section 7]{EGNO})
on the left on $\cO_0$ in the obvious way, and we can apply the theory of internal hom as in \cite{Os} (see \cite[\S 7.8-7.10]{EGNO}). 
This action restricts to the projective objects, that is, 
$\operatorname{proj}(\hcii)\cong \mathscr{P}$ 
acts on $\operatorname{proj}(\cO_0)$ as well as on $\cO_0$. 
The action of $\cP$ on $\cO_0$ is given by exact functors, and the 
action of $\hcii$ on $\cO_0$ is given by right exact functors. 
For $M\in\cO_0$, the internal hom functor $[M,-]:\cO_0\to \hcii$ 
is defined as the right adjoint of the right exact functor 
$- \otimes_U M : \hcii\to \cO_0$.
The object $[M,N]$, for $M,N\in \cO_0$, agrees with the subspace 
$\cL(M,N)$
of $\Hom_\mathbb{C}(M,N)$ where the action of $\mathfrak{g}$ is locally finite, as defined in \S \ref{s2.3}. 
The following adjunction confirms this fact.

\begin{proposition}\label{L=[]}
For all $M,N\in\cO_0$ and $X\in\hcii$, we have:
\[\Hom_{\cO}(X\otimes_U M,N )\cong \Hom_{\hcii}(X,\cL(M,N)).\]
\end{proposition}

\begin{proof}
The Hom-tensor adjunction gives
\[\Hom_{\cO}(X\otimes_U M,N)=\Hom_{U-}(X\otimes_U M,N)\cong \Hom_{U- U}(X,\Hom_{\mathbb C} (M,N)). \]
Now, since $X\in \hcii$, any $U- U$-bimodule map $X\to \Hom_{\mathbb C} (M,N)$ factors through $\cL(M,N)$, that is, we have 
\[\Hom_{\cO}(X\otimes_U M,N)\cong \Hom_{\hcii}(X,\cL (M,N)). \qedhere\]
\end{proof}

Note that $\cL(\Delta_e,\Delta_e)$, as well as $\cL(M,M)$, for any $M\in\cO_0$, 
is an algebra in $\hcii$. The multiplication coming from the internal hom 
construction and the multiplication restricted from $\Hom_\mathbb{C}(M,M)$ coincide. 
We denote by $\operatorname{mod}_{\hcii}(\cL(M,M))$ the category of (right) 
$\cL(M,M)$-modules in $\hcii$. We denote by $\operatorname{proj}_{\hcii}(\cL(M,M))$ 
the category of projective objects in $\operatorname{mod}_{\hcii}(\cL(M,M))$. The morphism 
spaces in both categories are denoted  $\Hom_{\cL(M,M)}(-,-)$. There is a natural action of $\hcii$ on $\operatorname{mod}_{\hcii}(\cL(M,M))$ on the left given by the
monoidal structure of $\hcii$. We refer to \cite{EGNO} for further details on
module categories in monoidal categories.

\begin{proposition}\label{ihom}
For $M\in\cO_0$, we have 
\[\Add_{\cO_0}\{\theta_x M\ |\ x\in W\}\cong \operatorname{proj}_{\hcii}(\cL(M,M))\]
as module categories over $\mathscr{P}$. Moreover, we have 
\[{\Add_{\cO_0}\{FM\ |\ F\in \hcii\}}\cong \operatorname{mod}_{\hcii}(\cL(M,M))\] 
 as module categories over $\hcii$. 
\end{proposition}

\begin{proof}
Since $\mathscr{P}\cong \operatorname{proj}\hcii$ has monoidal duals (i.e., rigid in the sense of \cite{EGNO}; it is also fiat in the sense of \cite{MM1}), the internal hom theory as in \cite[Section 7]{EGNO} and \cite{MMMT} applies.
Consider the functor 
\[\cL(M,-):\cO_0\to \hcii.\]
Objects of the form  $\cL(M,N)$ can be viewed as objects in  $\operatorname{mod}_{\hcii}(\cL(M,M))$ since composition defines an action $\cL(M,N)\otimes \cL(M, M)\to \cL(M,N)$. 
Now, consider the restriction $\overline \Phi$ of $\cL(M,-)$ as follows:  
\[\overline \Phi:{\Add_{\cO_0}\{FM\ |\ F\in \hcii\}}\to \operatorname{mod}_{\hcii}(\cL(M,M)).\]
Then we have that
\begin{enumerate}
     \item $\overline\Phi$ is a $\mathscr{P}$-module functor;
    \item $\cL(M,M)$ is projective in $\operatorname{mod_{\hcii}}(\cL(M,M))$.
\end{enumerate}
This yields a $\mathscr{P}$-module functor 
\[\Phi: {\Add_{\cO_0}\{\theta M\ |\ \theta\in \mathscr{P} \}}\to \operatorname{proj}_{\hcii}(\cL(M,M))\]
(see the proofs in \cite[\S 4.2]{MMMT} for the dual version).
Since $\Add_{\cO_0}\{\theta M\ |\ \theta\in \mathscr{P} \}$ is, by definition, a transitive representation of $\mathscr{P}$, an algebra analog of \cite[Theorem 4.7]{MMMT} gives the two equivalences in the statement.
\end{proof}

The following is a slight variation of \eqref{eq1}.

\begin{proposition}\label{HC}
We have $\mod_{\hcii}(\cL(\Delta_e,\Delta_e))\cong \cO_0$. 
\end{proposition}
\begin{proof}
By Proposition \ref{ihom}, we have equivalences 
\[\Add_{\cO_0}\{\theta_x\Delta_e\ |\ x\in W\}\cong \operatorname{proj}_{\hcii}(\cL(\Delta_e,\Delta_e))\]
and 
\begin{equation}\label{eqt}
  {\Add_{\cO_0}\{F\otimes_U \Delta_e\ |\ F\in \hcii\}}\cong \operatorname{mod}_{\hcii}(\cL(\Delta_e,\Delta_e)). 
\end{equation}
But $\theta_x\Delta_e\cong P_x$, and therefore
$\Add_{\cO_0}\{\theta_x\Delta_e\ |\ x\in W\}\cong \operatorname{proj}\cO_0.$ The claim follows.
\end{proof}

\subsection{The image of $\cL(L_w,L_w)$ in $\mathcal{O}_0$}\label{s3.2}

Graded characters of the projective modules $P_w$, $w\in W$, can be computed using Kazhdan-Lusztig combinatorics. Recall from Subsection \ref{ss:cella} the left $\leq_\mathtt{L}$, right $\leq_\mathtt{R}$ and the two-sided order $\leq_\mathtt{J}$ on $W$.

Note that $\mathrm{Ann}(L_x)=\mathrm{Ann}(L_y)$ if and only if 
$x\sim_\mathtt{L} y$, and moreover, the inclusion order of primitive ideals 
for simple modules in $\mathcal{O}_0$ is the opposite of $\leq_\mathtt{L}$, see \cite{BV1,BV2}.
In particular, the image of $U(\mathfrak{g})$
in $\cL(L_w,L_w)$ depends (up to isomorphism) only on the left cell of $w$.

Let $d\in W$ be a Duflo element. The module $\theta_d L_d$ is indecomposable
with simple top $L_d\langle\mathbf{a}(d)\rangle$ and simple socle
$L_d\langle-\mathbf{a}(d)\rangle$. It has a unique simple subquotient
isomorphic to the trivial $\mathfrak{g}$-module, and this subquotient is in
degree zero. Therefore there is a unique, up to scalar, non-zero homomorphism 
from $P_e$ to $\theta_d L_d$. We denote by $D_d$ the image of this homomorphism.
The module $D_d$ is indecomposable with simple top $L_e$ and simple socle
$L_d\langle-\mathbf{a}(d)\rangle$. All other simple subquotients of 
$D_d$ have the form $L_w\langle i\rangle$, where $-\mathbf{a}(d)<i<0$
and $x<_\mathtt{R} d$ (and hence also $x<_\mathtt{L} d$ since $d$ is an involution). 
We refer to \cite[Section~3]{Ma2} for details.

\begin{proposition}\label{prop5}
We have
\begin{displaymath}
(U(\mathfrak{g})/\mathrm{Ann}(L_d))\otimes_{U(\mathfrak{g})}P_e\cong D_d.
\end{displaymath}
\end{proposition}

\begin{proof}
The equivalence \eqref{eq1} sends, by construction, 
$U(\mathfrak{g})/(U(\mathfrak{g})\mathbf{m})$
to $P_e$ inducing a bijection between subobjects. Note that subobjects of the
former are exactly the two-sided ideals and that $D_d$ is, by definition, a quotient of $P_e$. 
Therefore, it is enough to argue that $\mathrm{Ann}(D_d)=\mathrm{Ann}(L_d)$.
This follows from \cite[Lemma~6]{KhM}.
\end{proof}

Denote by $\overline{D}_d$ the intersection of the kernels of all possible
homomorphisms 
\begin{displaymath}
\varphi:\theta_d L_d\to \theta_w L_d\langle i\rangle, \quad
w\sim_\mathtt{R} d,\quad i\in\mathbb{Z},
\end{displaymath}
satisfying $\varphi(D_d)=0$.
By construction, $D_d\subset \overline{D}_d$.
In particular, $L_d$ is the simple socle of $\overline{D}_d$
and all other composition factors of $\overline{D}_d$
are of the form $L_w\langle i\rangle$, where $-\mathbf{a}(d)<i$
and $w<_\mathtt{R} d$ (and also $w<_\mathtt{L} d$).

\begin{proposition}\label{prop6}
We have
\begin{displaymath}
\cL(L_d,L_d)\otimes_{U(\mathfrak{g})}P_e\cong \overline{D}_d.
\end{displaymath}
\end{proposition}

\begin{proof} 
By \eqref{eq1}, it is enough to prove that
\begin{equation}\label{eq2}
\cL(P_e,\overline{D}_d)\cong \cL(L_d,L_d).
\end{equation}
The natural projection $P_e\tto D_d$ induces an embedding
\begin{equation}\label{eq3}
\cL(D_d,\overline{D}_d)\subset\cL(P_e,\overline{D}_d).
\end{equation}
Let $K$ be the kernel of $P_e\tto D_d$.
From \cite[Corollary~3]{Ma2} it follows that 
$\theta K$ has no non-zero homomorphisms to $\overline{D}_d$,
for any projective functor $\theta$. Therefore
the inclusion \eqref{eq3} is, in fact, an isomorphism. 
Consequently, we have $\cL(P_e,\overline{D}_d)\cong \cL(D_d,\overline{D}_d)$.
Now \eqref{eq2} follows from
\cite[Lemma~11]{KhM} and \cite[Lemma~12]{KhM}.
\end{proof}

As an immediate corollary from Propositions~\ref{prop5} and \ref{prop6},
we have the following statement (which is a reformulation of the main result
of \cite{KhM}):

\begin{corollary}\label{cor8}
For $d\in W$ Duflo element, $\mathbf{K}(d)$ is equivalent to
$D_d=\overline{D}_d$.
\end{corollary}

\begin{example}\label{ex9}

In type $A_3$ with Dynkin diagram \begin{tikzpicture}[scale=0.4]
\draw (0 cm,0) -- (4 cm,0);
\draw[fill=white] (0 cm, 0 cm) circle (.15cm) node[above=1pt]{\scriptsize $1$};
\draw[fill=white] (2 cm, 0 cm) circle (.15cm) node[above=1pt]{\scriptsize $2$};
\draw[fill=white] (4 cm, 0 cm) circle (.15cm) node[above=1pt]{\scriptsize $3$};
\end{tikzpicture},
where $1,2,3$ are simple reflections, it is known from \cite{MS2}
that $\mathbf{K}(13)$ is false. In this case one can compute that
the graded composition multiplicities of $D_{13}$ and
$\overline{D}_{13}$ are, respectively, as follows (here $e$ is in degree zero,
and we also abbreviate $L_w$ by $w$, for simplicity): 
\begin{equation}\label{eqn1}
\xymatrix@C=1em@R=1em{
&e\ar@{-}[dl]\ar@{-}[dr]&\\
1\ar@{-}[dr]&&3\ar@{-}[dl]\\
&13&
}\quad\quad\quad\quad\quad\quad\quad\quad\quad
\xymatrix@C=1em@R=1em{
&&e\ar@{-}[dl]\ar@{-}[dr]&\\
123\ar@{-}[drr]&1\ar@{-}[dr]&&3\ar@{-}[dl]&321\ar@{-}[dll]\\
&&13&
}
\end{equation}

\end{example}

\subsection{$\mathbf{K}(d)$ implies $\mathbf{KMM}(*,d)$}\label{s3.3}

\begin{proposition}\label{prop7}
Let $d\in W$ be a Duflo element. Then 
$\mathbf{K}(d)$ implies $\mathbf{KMM}(*,d)$.
\end{proposition}

\begin{proof}
By Proposition~\ref{prop6}, the claim $\mathbf{K}(d)$
is equivalent to  $D_d=\overline{D}_d$.  
Note that $D_d$ is a quotient of $P_e$, and hence \eqref{eq9} and \eqref{eq9-1} implies that
$[D_d:L_x\langle -\mathbf{a}(x)\rangle]\leq 1$. This implies $\mathbf{KMM}(*,d)$.
\end{proof}

From Example~\ref{ex9} we see that $\mathbf{K}$ is strictly stronger
than $\mathbf{KMM}$. Indeed, the two additional elements on the 
right picture in \eqref{eqn1} (compared with the left picture) both have 
the $\mathbf{a}$-value $1$ and have multiplicity one (they are not Duflo elements either). Therefore, in this case, we have that
$\mathbf{K}(13)$ is false while $\mathbf{KMM}(*,13)$ is true.

\section{$\mathbf{K}$ vs.  $\mathbf{KM}$ and $\mathbf{Kh}$}\label{s30}

\subsection{Some computations in $\cO^{\hat{\mathcal{R}}}$}

\begin{lemma}\label{d unit}
Let $\mathcal{L}$ be a left cell in $W$ and $d\in \mathcal{H}=\mathcal{L}\cap \mathcal{L}\inv$
the corresponding Duflo element. 
For each $y\in \mathcal{L}$ and $y'\in \mathcal{L}\inv$, we have
\begin{equation}\label{eq:dy=y}
\theta_d\theta_y = \theta_y\langle \mathbf{a}(d)\rangle \oplus \theta,
\qquad
\theta_{y'}\theta_d = \theta_{y'}\langle \mathbf{a}(d)\rangle \oplus \theta',
\end{equation}
where $\theta,\theta'\in\mathscr{P}$ are (possibly empty) direct sums of (some) 
$\theta_z\langle a\rangle$ with $z\not\leq_\mathtt{J} d$ or $a<\mathbf{a}(d)$.
\end{lemma}

\begin{proof}
This follows from  Lemma \ref{[d unit]}.
\end{proof}

The following lemma slightly extends \cite[Theorem 6]{Ma2}. 

\begin{lemma}\label{xLyProj}
Let $\mathcal{R}\subseteq W$ be a right cell, 
$y \in \mathcal{R}$ and $x\in W$ be such that $x\sim_{\mathtt{J}}y$. 
Then $\theta_x L_y$ is a projective-injective object in $\cO^{\hat{\mathcal{R}}}$. 
\end{lemma}

\begin{proof}
If $y=d$ is Duflo, the claim follows directly from \cite[Theorem 6]{Ma2}. 
Moreover, the objects $\theta_x L_d$, where $x\inv \in \mathcal{R}$,
give a complete and irredundant list of indecomposable projective-injective
modules in $\cO^{\hat{\mathcal{R}}}$. Set $a:=\mathbf{a}(d)$. 
Each of $\theta_x L_d$ has simple top $L_x\langle a\rangle$
and simple socle $L_x\langle -a\rangle$ and no other subquotients 
in these extreme non-zero components.

For an arbitrary $y$, by \cite[Proposition~1]{Ma2}, 
the module $\theta_x L_y$ is either zero (in which case the claim is obvious) or
the top of $\theta_x L_y$ is concentrated in degree $-a$
and the socle of $\theta_x L_y$ is concentrated in degree $a$,
moreover, these top and socle have the same length
due to self-duality of $\theta_x L_y$. Therefore, from the previous
paragraph it follows that the minimal projective cover of $\theta_x L_y$
is an isomorphism. The claim of the lemma follows. We note that, alternatively, 
the claim of the lemma follows from \cite[Theorem~2]{KMMZ}.
\end{proof}

\begin{proposition}\label{dLy}
Let $\mathcal{L}$ be a left cell and $y,d\in \mathcal{L}$ 
be such that $d$ is Duflo. Let $d'$ be the Duflo element in the 
right cell $\mathcal{R}$ of $y$.
Then 
$\theta_dL_y\cong \theta_yL_{d'}\cong P^{\hat{\mathcal{R}}}_y\langle \mathbf{a}(d)\rangle$,
where $P^{\hat{\mathcal{R}}}_y$ denotes the projective cover of $L_y$ in $\cO_0^{\hat{\mathcal{R}}}$.
\end{proposition}

\begin{proof}
By Lemma \ref{d unit}, we have
\[\theta_d P_y \cong \theta_d\theta_y P_e \cong P_y\langle \mathbf{a}(d)\rangle \oplus P,\]
where $P$ is a direct sum of (some) $P_z\langle a \rangle$ with either $z\not\leq_\mathtt{J} d$ or $a<\mathbf{a}(d)$.
Applying $\theta_d$ to the canonical map $P_y\surj L_y$, we get
\begin{equation}\label{ala}
    P_y\langle \mathbf{a}(d)\rangle \oplus P \cong \theta_d P_y \surj \theta_d L_y.
\end{equation} 
By \cite[Proposition~1]{Ma2},
the top of $\theta_d L_y$ is concentrated in degree 
$\mathbf{a}(d)=\mathbf{b}(d,y)=\mathbf{a}(y)$.
Therefore, this top is a direct sum of 
$L_w\langle \mathbf{a}(d)\rangle$, for $w\sim_\mathtt{R} y$. 
It follows that the image of any possible map from $P$ to $\theta_d L_y$ does not contain any
top component of $\theta_dL_y$, showing that \eqref{ala} restricts to 
$P_y\langle \mathbf{a}(d)\rangle \surj \theta_d L_y$.
Now the claim of the proposition follows combining the facts that 
$\theta_dL_y$ is projective in $\cO_0^{\hat{\mathcal{R}}}$,
see Lemma~\ref{xLyProj}, and that $\theta_y L_{d'}$ is the projective 
cover of $L_y\langle \mathbf{a}(d)\rangle$ in $\cO_0^{\hat{\mathcal{R}}}$, see 
\cite[Theorem~6]{Ma2}. 
\end{proof}

\begin{proposition}\label{b=c}
Let  $\mathcal{R}\subseteq W$ be a right cell  and $d,y\in \mathcal{R}$
with $d$ Duflo. Let $\mathcal{L}$ be the left cell of $y$
and $x\in W$ be such that $x\sim_{\mathtt{J}}y$. Then
$\theta_{x}L_{y}=0$ if $x\inv\not\in \mathcal{L}$
and, in case $x\inv\in \mathcal{L}$, we have 
\begin{equation}\label{cz}
\theta_{x}L_{y} \cong \bigoplus_{z\in \mathcal{L}\cap\mathcal{R}}
\theta_zL_d^{\oplus \gamma_{y,x,z\inv}}.
\end{equation}
\end{proposition}

\begin{proof}
That $\theta_{x}L_{y}=0$, for all $x\in W$ such that $x\sim_{\mathtt{J}}y$
and $x\inv\not\in \mathcal{L}$ follows from \S~\ref{s2.1}.
From \cite[Proposition~1 and Theorem~6]{Ma2} it follows that 
$\theta_{x}L_{y}$ is a direct sum of $\theta_zL_d$, where
$z\in \mathcal{L}\cap\mathcal{R}$, with
some multiplicities which we denote by $c_z\in\mathbb Z_{\geq 0}$. Applying $\theta_x$ to 
$\theta_y L_d\langle-\mathbf{a}(d)\rangle\tto L_y$ and using
\cite[Proposition~1]{Ma2}  and \S~\ref{ss:asym}, we obtain $c_z=\gamma_{y,x,z\inv}$. 
\end{proof}

\subsection{Computation with simple bimodules}

Denote by $\beta_x$ the simple Harish-Chandra bimodule corresponding 
to $x\in W$, that is, $\beta_x$ is the simple top of $\theta_x$ in $\hcii$. 
Denote by $\star:\hcii\to \hcii$ the simple preserving duality satisfying 
$(\beta \langle a \rangle)^\star=\beta^\star\langle -a \rangle$, see \cite[\S~4.1]{MM0}. 
This duality restricts to $\cO_0$ and is compatible with the action
in the sense that, for $M\in \cO_0$ and 
$\theta \in \hcii$, we have a natural isomorphism $\theta^{\star} M^{\star}
\cong (\theta M)^{\star}$.
The following statement provides a different interpretation of the internal hom.

\begin{proposition}\label{betatensor}
For any $y\in W$, we have $\cL(L_y,({}_-)^{\star}I_e)^\star\cong  - \otimes_U\beta_{y\inv}$.
\end{proposition}

\begin{proof}
Consider the functor $-\otimes_U \beta_{y\inv}:\hcii\to \hcii$. 
It naturally commutes (in the sense of \cite{Kho}) with the left action of $\hcii$
on itself. Therefore, by the main result of \cite{Kho}, it is determined, up to 
isomorphism, by its value at $\theta_e$. Clearly, this value is
$\theta_e \otimes_U \beta_{y\inv}  = \beta_{y\inv}$.

By construction, $\cL(L_y,({}_-)^{\star}I_e)^\star$ is also right exact and
naturally commutes with the left action of $\hcii$ by \cite[\S~6.8]{Ja}.
The value 
\begin{displaymath} 
\cL(L_y,(\theta_e)^{\star}I_e)^\star\cong
\cL(L_y,(\theta_eP_e)^{\star})^\star\cong
\cL(L_y,I_e)^\star
\end{displaymath}
at $\theta_e$ can be identified studying homomorphisms from projective objects:
\[\Hom_{\hcii}(\theta_w, \cL(L_y,I_e)) \cong \Hom_{\cO}(\theta_w L_y, I_e) \cong \Hom_\cO(L_y,\theta_{w\inv} I_e) \cong\Hom_\cO(L_y,I_{w\inv})=\delta_{y,w\inv}\mathbb C. \]
It follows that $\cL(L_y,I_e)\cong \cL(L_y,I_e)^\star\cong \beta_{y\inv}$.
The claim follows.
\end{proof}

Let $y,z\in W$. Evaluating both sides of Proposition~\ref{betatensor} at $\beta_z$
and noticing that all simples are self-dual, we obtain
\begin{equation}\label{eq:bebe}
  \mathcal{L}(L_y,L_z)\cong (\beta_z\beta_{y\inv})^{\star}.   
\end{equation}
Evaluating both sides of Proposition~\ref{betatensor} at $\theta_z$, we obtain
\begin{equation}\label{eq:thebe}
  \mathcal{L}(L_y,I_z)\cong (\theta_z\beta_{y\inv})^{\star}.    
\end{equation} 

\begin{corollary}\label{subz}
We have the canonical inclusion 
$\cL(L_y,L_z)\otimes_U P_e \inj \theta_zL_{y\inv}$, for $y,z\in W$. 
\end{corollary}

\begin{proof}
Applying the left exact functor
$\cL(L_y,{}_-)$ to the canonical inclusion $L_z\inj I_z$, we obtain
\begin{displaymath}
\cL(L_y,L_z)\inj  \cL(L_y,I_z)
\end{displaymath}
Applying now the equivalence ${}_-\otimes_U P_e$, the right term becomes
$(\theta_zL_{y\inv})^{\star}\cong \theta_zL_{y\inv}$.
\end{proof}

Note that $\cL(L_y,L_z)=0$ unless $y\sim_\mathtt{R} z$. 
Proposition~\ref{betatensor} provides the following 
description of $\cL(L_y,L_z)$ in the case when it is nonzero:

\begin{proposition}\label{present}
Let $y,z\in W$ be such that $y \sim_\mathtt{R} z$. 
Then the object $\cL(L_y,L_z)\otimes_U P_e$ admits a copresentation
\begin{equation}\label{9}
0\to \cL(L_y,L_z)\otimes_U P_e\to \theta_zL_{y\inv} \to 
\bigoplus_{w\in W} \boldsymbol{\mu}(z,w)\theta_wL_{y\inv}\langle 1\rangle.
\end{equation}
\end{proposition}

\begin{proof}
The simple object $L_z\in\cO$ admits an injective copresentation
\begin{equation}\label{inresol}
L_z\inj I_z\to \bigoplus_{w\in W} \boldsymbol{\mu}(z,w) I_w 
\langle 1\rangle .
\end{equation}
Applying the left exact functor $\cL (L_y,-)\otimes_U P_e$ to \eqref{inresol}, 
we obtain \eqref{9} using \eqref{eq:thebe} 
by the same arguments as in Corollary~\ref{subz}.
\end{proof}

Next lemma give us additional information on the copresentation in \eqref{9}.

\begin{lemma}\label{socE}
For $y,z\in W$ such that $y\sim_\mathtt{R} z$, 
we have $\soc\cL(L_y,L_z) = \soc\cL(L_y,I_z)$.
\end{lemma}

\begin{proof}
We prove the equivalent statement in $\cO$, namely, 
$\soc\cL(L_y,L_z)\otimes_U P_e = \soc\cL(L_y,I_z)\otimes_U P_e$.
Since $y\sim_\mathtt{R} z$, the socle of $\cL(L_y,I_z)\otimes_U P_e\cong \theta_zL_{y\inv}$ 
is concentrated in degree $\mathbf{a}(z)$. 
By Proposition \ref{present}, it is thus enough to note that each $\theta_w L_{y\inv}\langle 1\rangle$ has extreme degree $\mathbf{b}(w,y\inv)-1\leq \mathbf{a}(y\inv)-1<\mathbf{a}(z)$,
see Proposition~\ref{propnn125}.
\end{proof}

From Proposition~\ref{present} and Theorem~\ref{thm-kmm-original}, we have that,
for $y,z\in W$ such that $y\sim_\mathtt{R} z$,
the assumption $$[\cL(L_y,L_z)\otimes_U P_e:L_w\langle -a \rangle]\neq 0,$$
for some $w\in W$ and $a\in\mathbb{Z}$, implies  $\mathbf{a}(w)\leq a\leq \mathbf{a}(y)$.

The following statement generalizes \cite[Lemma~8(i)]{KhM}.

\begin{corollary}\label{E_yz}
Let $y,z\in W$ be such that $y\sim_{\mathtt{R}} z$, and 
$\mathcal{R}$ be the right cell of $y^{-1}$.
Then $\cL(L_y,L_z)\otimes_U P_e$ is isomorphic to the largest submodule 
$M$ of $\theta_zL_{y\inv}\in \cO_0^{\mathcal{\hat{R}}}$ such that 
any simple subquotient of $M/\soc M $ is, up to shift of grading, of the form 
$L_w$, for some $w\not\in\mathcal{R}$. 
\end{corollary}

\begin{proof}
We start by showing that, for $w\in \mathcal{R}$, the module $L_w$ can only appear in the
socle of $\cL(L_y,L_z)\otimes_U P_e$, that is in degree $\mathbf{a}(w)$. This
follows directly from \cite[Proposition~1]{Ma2} and the adjunction
\begin{displaymath}
\Hom_{\cO}(\theta_w L_y,L_z )\cong \Hom_{\hcii}(\theta_w,\cL(L_y,L_z)),
\end{displaymath}
given by Proposition~\ref{L=[]}. Therefore $\cL(L_y,L_z)\otimes_U P_e\subset M$
by Corollary~\ref{subz}. Now, the necessary isomorphism 
$\cL(L_y,L_z)\otimes_U P_e=M$ follows from Proposition~\ref{present} since
the socle of each summand in the second term of the injective 
(in $\cO_0^{\mathcal{\hat{R}}}$) copresentation of
$\cL(L_y,L_z)\otimes_U P_e$ is, up to shift of grading, of the form 
$L_w$, for some $w\in\mathcal{R}$. 
\end{proof}

For a Duflo element $d\in W$, let $\mathcal{R}$ be the right cell  of $d$.
Denote by $\Psi:\mathcal{O}_0^{\mathcal{\hat{R}}}\to \mathcal{O}_0^{\mathcal{\hat{R}}}$
the functor of partial approximation with respect to projective-injective
modules in $\mathcal{O}_0^{\mathcal{\hat{R}}}$, see \cite[\S~2.4]{KM}.
The easiest way to define this functor is as follows: we let
$\mathcal{X}$ denote the Serre subcategory of $\mathcal{O}_0^{\mathcal{\hat{R}}}$
generated by all simple which do not appear in the socle of
projective-injective objects. Then $\Psi$ is the composition of
the (exact) natural projection $\mathcal{O}_0^{\mathcal{\hat{R}}}\tto 
\mathcal{O}_0^{\mathcal{\hat{R}}}/\mathcal{X}$ followed by the right adjoint
of this projection. In particular, $\Psi$ is left exact and is equipped with a
natural transformation $\eta$ from the identity to it which is non-zero exactly on
those simple modules  which appear in the socle of a projective-injective module.
By \cite[Theorem~6]{Ma2},
Corollary~\ref{E_yz} says precisely that 
$\cL(L_y,L_z)\otimes_U P_e$ is isomorphic to 
$\Psi(\soc(\theta_zL_{y\inv}))$.

We can now relate our discussion closer to the results of \cite{KhM}.

\begin{proposition}\label{tohatR}
Let $y\in W$ and $\mathcal{R}$ be the right cell containing $y\inv$. 
Then the left exact functor 
$$\cL(L_y,{}_-)\otimes_U P_e:\cO_0\to \cO_0$$ 
has image in $\cO^{\hat{\mathcal{R}}}$ and 
maps $I_z\in \cO_0$, for $w\in W$, to zero unless $z\leq_\mathtt{R} y$. 
\end{proposition}

\begin{proof}
Note that, for any $M\in \cO_0$, for $L_x$ to appear in the composition series of $\cL(L_y,M)\otimes_U P_e$, the space
\begin{equation}\label{lnn137}
\Hom_{\cO}(P_x,\cL(L_y,M)\otimes_U P_e) \cong 
\Hom_{\hcii}(\theta_x,\cL(L_y,M))\cong \Hom_{\cO}(\theta_x L_y,M)
\end{equation}
should be nonzero, which requires $\theta_x L_y\neq 0$ and thus $x\leq_\mathtt{R} y\inv$.
This proves the first claim. 

Note that all simple subquotients of $\theta_x L_y$ have the form
$L_z$, where $z\leq_\mathtt{R} y$. This implies the second claim.
\end{proof}

For $y\in W$, denote by  $\mathcal{R}$ the right cell of $y$
and by $\mathcal{R}'$ the right cell of $y\inv$. 
Proposition~\ref{tohatR} says that the functor
$\cL(L_y,-)\otimes_U P_e$ restricts to  the functor
\begin{equation}
\Psi_y:\cO_0^{\hat{\mathcal{R}}}\to \cO_0^{\hat{\mathcal{R}'}}.
\end{equation}
If $y$ is a Duflo element, then $\mathcal{R}=\mathcal{R}'$ 
and the functor $\Psi_y$ agrees with the functor of 
partial coapproximation with respect to the projective-injective
modules in $\cO_0^{\hat{\mathcal{R}}}$, cf. \cite[Corollary~7.21]{KhM}.

\subsection{Kostant's problem via internal hom}

By \cite[\S~6.9]{Ja}, we have $U/\ann(P_e)\cong \cL(P_e,P_e)$
as algebras in $\hcii$. Recall that $\ann(P_e)=UI$. 
More generally, for every $M\in\mathcal{O}_0$,
there is an injective map of algebras 
$\phi_M:U/\ann(M)\to  \cL(M,M)$. In the case $M=L(w)$, since 
$U/\ann(P_e)\surj U/\ann(L_w)$, we have an algebra homomorphism 
\begin{equation}
    \cL(P_e,P_e)\to \cL(L_w,L_w)
\end{equation}
which is surjective if and only if $\mathbf{K}(w)$ holds.

Fix $y\in W$, let $\mathcal{L}$ be the left cell of $y$ and $\mathcal{R}=\mathcal{L}\inv$, 
the right cell of $y\inv$. For simplicity, we write $U_y:=U/\ann(L_y)$ and $A_y:=\cL(L_y,L_y)$, and let
$\phi_y:=\phi_{L_y}:U_y \inj A_y$.
This induces the restriction functor
\[\res^{A_y}_{U_y}:\operatorname{mod}_{\hcii}(A_y)\to \operatorname{mod}_{\hcii}(U_y).\]

\begin{proposition}\label{Uytriv}
The category $\operatorname{mod}_{\hcii}(U_y)$ can be identified with 
$^\infty_0\mathcal H^{\ann(L_y)}_0$, i.e., the category of Harish-Chandra 
bimodules that are annihilated by $\ann(L_y)$ on the right. 
\end{proposition}

\begin{proof}
Note that ${}_-\otimes_U U_y$ is the identity functor on the full subcategory  of
$\hcii$ consisting of all $M$ such that $M\cdot\ann(L_y)=0$. 
In particular, we have $\mu:U_y\otimes_U U_y \cong U_y$ and this, 
together with the unit map $\epsilon:U/UI\surj U_y$, makes $U_y$ an 
algebra object in $\hcii$. For this algebra object, the 
corresponding module category $\operatorname{mod}_{\hcii}(U_y)$ 
is a full  subcategory of $\hcii$.  
The objects in $\operatorname{mod}_{\hcii}(U_y)$ are exactly those $M\in\hcii$ 
for which  
$M\otimes_U U_y\cong M$, i.e., $ M\cdot\ann (L_y)= 0$. 
In fact, if $M\otimes_U U_y\cong M$, then $1_M\otimes_U \mu$ 
defines a $U_y$-module structure on $M$, and, if we have a 
$U_y$-modules structure $m:M\otimes_U U_y \to M$, then the 
unit law gives $1_M= m\circ (1_M\otimes_U \epsilon)$, which 
implies that $m$ is an isomorphism since $ 1_M\otimes_U \epsilon$ is surjective.
The claim follows.
\end{proof}

Recall that, for an additive category $\mathcal{A}$, its projective abelianization is denoted $\overline{\mathcal{A}}$, see \cite[\S~3.5]{MM1}.

\begin{proposition}\label{Uymod}
We have $\operatorname{mod}_{\hcii}(U_y)\cong
\overline{\Add_{\hcii}\{\theta_x\otimes_U U_y\,:\,x\in W\}} 
\cong \cO^{\mathcal{\hat R}}$.
\end{proposition}

\begin{proof}
Since $\beta_x$,  for $x\in W$ such that $x\leq_\mathtt{R} y\inv$, are exactly 
the simples that are killed by $\ann(L_y)$ on the right, the equivalence 
${}_-\otimes_U P_e$
given by \eqref{eq1} restricts to the desired equivalence via Proposition~\ref{Uytriv}.
\end{proof}

\begin{corollary}\label{projU}
The indecomposable projectives in $\operatorname{mod}_{\hcii}(U_y)$ are exactly $\theta_x\otimes_U U_y$, for $x\in\hat{\mathcal{R}}$.
\end{corollary}

\begin{proof}
This follows from the fact that the nonzero $\theta_x\otimes_U U_y$ corresponds
in $\cO_0$ to a quotient of  $P_x$. Therefore such modules are indecomposable 
and mutually non-isomorphic.
\end{proof}

\begin{proposition}\label{K to JKM}
For $y\in W$, we have that $\mathbf{K}(y)$ 
implies both $\mathbf{Kh}(y)$ and $\mathbf{KM}(*,y)$.
\end{proposition}

\begin{proof}
Suppose $\mathbf{K}(y)$ is true. Then we have 
\begin{equation}\label{eqnn652}
\operatorname{proj}_{\hcii}(A_y)\cong \operatorname{proj}_{\hcii}(U_y),
\end{equation}
where $\theta_x \otimes_U A_y$ corresponds to $\theta_x \otimes_U U_y$.
By Proposition \ref{ihom}, we also have
\begin{equation}\label{eqnn653}
\Add_{\cO_0}\{\theta_x L_y\,:\,x\in W\}\cong\operatorname{proj}_{\hcii}(A_y),
\end{equation}
where $\theta_xL_y$ corresponds to $\theta_x A_y$.
Combining \eqref{eqnn652} and \eqref{eqnn653}, 
we get an equivalence that identifies $\theta_xL_y$ with $\theta_x\otimes_U U_y \in  \hcii$. 
Thus, by Corollary \ref{projU}, $\theta_xL_y$ are indecomposable and mutually non-isomorphic.
\end{proof}

{\color{red} The proof of the following proposition contains a gap, for
corrigendum and addendum, see \cite{CA26}.}

\begin{proposition}\label{Morita}
For a Duflo element $d\in W$, if both $\mathbf{Kh}(d)$ and $\mathbf{KM}(*,d)$ are true, 
then  $\res^{A_d}_{U_d}$ is an equivalence.
\end{proposition}

\begin{proof}
Let $\mathcal{R}$  be the right cell of $d$.
The equivalence in Proposition~\ref{Uymod} restricts to the equivalence
\begin{equation}\label{eqnn654}
\operatorname{proj}(\cO^{\mathcal{\hat R}})\cong\operatorname{proj}_{\hcii}(U_d),
\end{equation}
which sends $P_w^{\mathcal{\hat R}}$, for $w\in W$, to $\theta_w\otimes_U U_d$.
By \cite[Corollary~3]{Ma2}, we have $L_d \hookrightarrow P_e^{\mathcal{\hat R}}$
and any simple subquotient of the cokernel of this inclusion is of the form
$L_x$, for $x<_{\mathtt{R}}d$, up to shift of grading.
For $w\in W$, applying $\theta_w$ gives the inclusion 
$\theta_w  L_d\hookrightarrow P_w^{\mathcal{\hat R}}$ such that 
any simple subquotient of the cokernel of this inclusion is of the form
$L_x$, for $x<_{\mathtt{R}}d$, up to shift of grading.
In particular, there are no homomorphisms from any such cokernel to any
object in $\operatorname{proj}(\cO^{\mathcal{\hat R}})$.
This means that, for $w,u\in W$,  any non-zero homomorphism from $P_w^{\mathcal{\hat R}}$
to $P_u^{\mathcal{\hat R}}$ restricts to a non-zero homomorphism
from $\theta_w L_d$ to $\theta_u L_d$. Using \eqref{eqnn653}
and \eqref{eqnn654}, this gives a faithful functor
\begin{displaymath}
\Upsilon:\operatorname{proj}_{\hcii}(U_d) 
\to \operatorname{proj}_{\hcii}(A_d).
\end{displaymath}
Because of $\mathbf{Kh}(d)$ and $\mathbf{KM}(*,d)$, this functor sends 
(pairwise non-isomorphic) indecomposable objects
to (pairwise non-isomorphic) indecomposable objects.

Using \eqref{eqnn653}, Proposition \ref{Uymod} and Proposition \ref{ihom}, the functor 
$\res^{A_d}_{U_d}$ gives rise to the faithful functor 
\[\Theta:\Add_\cO\{\theta_xL_d\ |\ x\in\hat{\mathcal{R}}\} \to 
\Add_{\hcii}\{  \theta_x\otimes_U  U_d\ |\ x\in\hat{\mathcal{R}}\}\]
which sends $\theta_xL_d$ to $\theta_x\otimes_U U_d$. The combination of
$\mathbf{Kh}(d)$ and $\mathbf{KM}(*,d)$ implies that  
$\Theta$ sends (pairwise non-isomorphic) indecomposable objects to 
(pairwise non-isomorphic) indecomposable objects.
Putting the faithful functors $\Upsilon$ and $\Theta$ together,
we thus conclude that they are equivalences. This means that 
$\res^{A_d}_{U_d}$ is an equivalence when restricted to projective objects
and hence is an equivalence. This completes the proof.
\end{proof}
 
\begin{theorem}\label{theorem}
Let $d\in W$ be a Duflo element. Then $\mathbf{K}(d)$ is true
if and only if both $\mathbf{Kh}(d)$ and $\mathbf{KM}(*,d)$ are true.
\end{theorem}

\begin{proof}
The ``only if" direction is Proposition~\ref{K to JKM}.
For the ``if" direction, assume that both $\mathbf{Kh}(d)$ and $\mathbf{KM}(*,d)$ are true.
Then, by Proposition~\ref{Morita}, $\res^{A_d}_{U_d}$ is an equivalence
sending $A_d$ to $U_d$. Therefore $A_d$ and $U_d$ are isomorphic as
objects in $\hcii$, which is exactly the claim $\mathbf{K}(d)$.
\end{proof}

\subsection{More on $\mathbf K$, $\mathbf{Kh}$, and $\mathbf{KM}$}

For $y\in W$, we set $E_y:= \cL(L_y,L_y)\otimes_U P_e \cong \overline{D_y}$
and $D_y := U_y\otimes_U P_e \in \cO$. 
Note that $D_y$ depends only on the left cell of $y$.
From $U_y\to \cL(L_y,L_y)$, we have  
$D_y\inj E_y$.

\begin{proposition}\label{propn-86}
Let $y\in W$ and $d$ be the Duflo element in the left cell of $y$.
Then $E_d$ is a summand of $E_y$. 
\end{proposition}

\begin{proof}
Let $\mathcal{R}$ be the right cell of $d$.
By adjunction, we have 
\begin{displaymath}
\mathrm{Hom}(L_d,\theta_y L_{y\inv})=
\mathrm{Hom}(\theta_{y\inv}L_d, L_{y\inv}).
\end{displaymath}
The right hand side has dimension one by \cite[Theorem~6]{Ma2}.
Therefore $\theta_d L_d$, which is the indecomposable injective
envelope of $L_d$ in $\mathcal{O}_0^{\mathcal{\hat{R}}}$, 
is a summand of the injective module $\theta_y L_{y\inv}$
in $\mathcal{O}_0^{\mathcal{\hat{R}}}$ with multiplicity one.

Now we can use Proposition~\ref{present} and Corollary~\ref{E_yz}.
The socle  of the module $E_d$ coincides with the socle of $\theta_d L_d$
and $E_d\cong \Psi(\mathrm{Soc}(\theta_d L_d))$.
The socle of $E_y$ coincides with the socle of $\theta_y L_{y\inv}$
and we have $E_y\cong \Psi(\mathrm{Soc}(\theta_y L_{y\inv}))$.
Now the claim of the proposition follows from the additivity of
$\Psi$ be applying the latter to the unique up to a scalar
split injection of $\mathrm{Soc}(\theta_d L_d)$
in $\mathrm{Soc}(\theta_y L_{y\inv})$ given by the previous paragraph.
\end{proof}

We summarize the above discussions in the following statement.

\begin{theorem}\label{DEthet}
Let $y\in W$ and $d$ be the Duflo element in the left cell of $y$, and denote by $\mathcal{H}$ the $\mathtt{H}$-cell of $y$. Then we have a commutative diagram
\begin{equation}\label{diag:DE}
    \begin{tikzcd}
        D_y\arrow[r,hook]& E_y \arrow[r,"",hook] &\theta_yL_{y\inv} \cong \bigoplus_{h \in \mathcal{H}} \theta_hL_{y\inv}\ ^{\oplus \gamma_{y\inv,y,h\inv}} \\
        D_d\arrow[r,hook]\arrow[u,equal]  &E_d \arrow[r,hook]\arrow[u,hook]& \theta_d L_d\arrow[u,hook],
    \end{tikzcd}
\end{equation}
whose vertical maps are split. 
Moreover, the diagram \eqref{diag:DE} restricts to
\begin{equation}\label{diag:socDE}
    \begin{tikzcd}
        L_d= \soc D_y\arrow[r,hook]& \soc E_y \arrow[r,equal] &\soc \theta_yL_{y\inv}  \\
        L_d= \soc D_d\arrow[r,equal]\arrow[u,equal]  &\soc E_d \arrow[r,equal]\arrow[u,hook]& \soc \theta_d L_d\arrow[u,hook]
    \end{tikzcd}
\end{equation}
and the object $E_y$ is determined by its socle.
\end{theorem}
\begin{proof}
The first claim is given by Proposition \ref{b=c} , Corollary \ref{subz}, and Proposition \ref{propn-86}.
The second claim is given by Lemma \ref{socE} and Corollary \ref{E_yz}.
\end{proof}

\begin{corollary}\label{of con1}
Let $d\in W$ be a Duflo element and $\mathcal{L}$ its left cell.
\begin{enumerate}
\item If $\mathbf{K}(d)$ is not true, then $\mathbf{K}(y)$ is not true,
for all $y\in \mathcal{L}$.
\item If $\mathbf{K}(d)$ is true, then
$\mathbf{K}(y)$ is equivalent to $\mathbf{KM}(y,y^{-1})$, for $y\in \mathcal{L}$.
\end{enumerate}
\end{corollary}

\begin{proof}
This is a direct consequence of Proposition~\ref{propn-86} and 
Theorem~\ref{theorem}.
\end{proof}

\begin{proposition}\label{prop-n942}
Let $d$ be the Duflo element in a left cell $\mathcal{L}$ 
in $W$ and $y,z\in \mathcal{L}$. 
If $t_y t_{y\inv} = t_z t_{z\inv}$, 
then, for each $x\in W$, we have 
$\mathbf{KM}(x,y) \Leftrightarrow \mathbf{KM}(x,z)$ 
and $\mathbf{Kh}(y)\Leftrightarrow \mathbf{Kh}(z)$. 
In particular, both $\mathbf{KM}(\ast,{}_-)$ and $\mathbf{Kh}({}_-)$ 
are constant on strongly regular left cells.
\end{proposition}

\begin{proof}
By Proposition~\ref{propn-86}, the assumption 
$t_y t_{y\inv} =t_z t_{z\inv}$ implies $E_y=E_z$. 
Thus, by Proposition~\ref{prop6} and   
Proposition~\ref{ihom}, we have equivalences
\[ \Add\{\theta_xL_y \colon x \in W\} \cong 
\operatorname{proj}_{\hcii}(A_y) = 
\operatorname{proj}_{\hcii}(A_z) 
\cong \Add\{\theta_xL_z \colon x \in W\}, \]
under which $\theta_xL_y\mapsto \theta_x L_z$, for each $x$.
The claim follows.
\end{proof}

At this point we can strengthen Theorem~\ref{theorem}.

\begin{theorem}\label{theorem1}
Let $y\in W$. Then $\mathbf{K}(y)$ is true
if and only if both $\mathbf{Kh}(y)$ and $\mathbf{KM}(*,y)$ are true.
\end{theorem}

\begin{proof}
The ``only if" direction is Proposition~\ref{K to JKM}.
For the ``if" direction, assume that both $\mathbf{Kh}(y)$ and $\mathbf{KM}(*,y)$ are true. Let $d'$ be the Duflo element in the right cell
of $y$ and $d$ be the Duflo element in the left cell of $y$.
We will first prove that $\mathbf{K}(d) \Leftrightarrow \mathbf{K}(y)$.

Because of $\mathbf{KM}(*,y)$, the module $\theta_{y\inv}L_y$
is indecomposable and hence is isomorphic to $\theta_{d'} L_{d'}$. 
In particular, $t_{y\inv}t_y=t_{d'}$ in the asymptotic ring $A(W)$, see \eqref{cz}. We
claim that this implies $t_yt_{y\inv}=t_d=t_dt_d$. Indeed,
since $t_{d'}$ is a local identity,
$t_yt_{y\inv}$ is an idempotent of $A(W)$.
Using \eqref{t't}, we can write
$t_yt_{y\inv}$ as $t_d+x$, where $x$ is
a linear combination of various $t_u$ with non-negative
integer coefficients. 
Then 
\[t_d+x=(t_d+x)(t_d+x)=t_d+t_dx+xt_d+x^2=t_d+2x+x^2\]
by Lemma \ref{[d unit]}, and thus $0=x+x^2$. Since the structure coefficients with respect to the $t_u$'s are non-negative, we conclude $x=0$
and $t_yt_{y\inv}=t_d$. By Proposition~\ref{propn-86}, this implies $E_y=E_d$, and hence we have $\mathbf{K}(d) \Leftrightarrow \mathbf{K}(y)$.

Now, by Proposition~\ref{prop-n942}, $\mathbf{Kh}(y)$ 
implies $\mathbf{Kh}(d)$,  and $\mathbf{KM}(*,y)$
implies $\mathbf{KM}(*,d)$. Therefore $\mathbf{K}(d)$
holds by Theorem~\ref{theorem} and we are done.
\end{proof}

In most small rank examples, we have $\mathbf{K}(y)$ if and only if $\mathbf{Kh}(y)$.
However, the situation is slightly more complicated in type  $G_2$.

\begin{example}\label{G2}
Let $\mathfrak{g}$ be of type $G_2$, and let $(W,S)$ be its Weyl group. Write $S=\{1,2\}$. The two-sided cells in $W$ are given by
\[ \begin{tabular}{|l|}
\hline
$e$ \\ \hline
\end{tabular} \qquad\qquad \begin{tabular}{|l|l|}
\hline
\begin{tabular}{@{}l@{}}$1$ \\ $121$ \\ $12121$\end{tabular} & \begin{tabular}{@{}l@{}}$21$ \\ $2121$\end{tabular} \\ \hline
\begin{tabular}{@{}l@{}}$12$ \\ $1212$\end{tabular} & \begin{tabular}{@{}l@{}}$2$ \\ $212$ \\ $21212$ \end{tabular} \\ \hline
\end{tabular} \qquad\qquad \begin{tabular}{|l|}
\hline
$212121$ \\ \hline
\end{tabular} \]
where the rows are left, and the columns are right cells. For the right cell $\mathcal{R}$ containing $1$, the indecomposables in $\Add\{\theta_xL_y\,:\, x\inv,y\in R\}$ are exactly 
\[\theta_1L_1,\ \ \theta_1L_{12},\ \ \theta_1L_{121},\ \ \theta_1L_{1212},\ \  \theta_1L_{12121}.\]
We describe the situation for $y\in \mathcal{R}$ in Table \ref{table:G2}, the rest is either trivial or symmetric to what is given. Note that, for dihedral types, $[\theta_xL_y]$ determines the isomorphism class of $\theta_xL_y$.
\begin{table}[ht]
\centering 
\begin{tabular}{c c c c c c} 
\hline
$y$ & $\mathbf{K}(y)$ & $\mathbf{KM}(*,y)$ & $\mathbf{Kh}(y)$ & $\theta_{y\inv} L_y$ & $E_y$ \\ [0.5ex] 
\hline 
1 & True & True & True & $\theta_1L_1$ & $D_1$\\ 
12 & False & False & True & $\theta_1L_1\oplus\theta_1L_{121} $ & $D_2\oplus L_{212}\langle -1\rangle$ \\
121 & False & False & False & $\theta_1L_1\oplus \theta_1L_{121}\oplus \theta_1L_{12121}$ & $D_1\oplus L_{121}\langle -1\rangle\oplus L_{12121}\langle -1\rangle$ \\
1212 & False & False & True & $\theta_1L_1\oplus \theta_1L_{121}$ & $D_2\oplus L_{212}\langle -1\rangle$\\
12121 & True & True & True & $\theta_1L_1$ & $D_1$\\ [1ex] 
\hline \\ 
\end{tabular}
\caption{$\mathbf{K}$, $\mathbf{KM}$ and $\mathbf{Kh}$ for $G_2$}
\label{table:G2}
\end{table}
\end{example}

\subsection{A combinatorial statement}

Although we reduce $\mathbf{K}(y)$ to $\mathbf{Kh}(y)$ and $\mathbf{KM}(*,y)$, it is not easy to determine $\mathbf{Kh}(y)$ and $\mathbf{KM}(*,y)$ in general. 
But the  combinatorial version $[\mathbf{Kh}](y)$ of $\mathbf{Kh}(y)$ can be determined by computing $h_{x,y,z}$ (see \eqref{gdim=h}).
Recall that Conjecture \ref{kconj-intro} claims $[\mathbf{Kh}](d)=\mathbf{K}(d)$ for Duflo elements $d\in W$. When this is true, the problem $\mathbf{K}(y)$ has a purely combinatorial solution as follows.

\begin{corollary}\label{cor1.4}
Let $d$ be a Duflo element and let $y\sim_\mathtt{L} d$. Suppose $[\mathbf{Kh}(d)]=\mathbf{K}(d)$. Then the condition $\mathbf{K}(y)$ is equivalent to the conjunction of the following conditions:
\begin{enumerate}
    \item \label{E=E} $t_yt_{y\inv}=t_d$ in $A(W)$;
    \item \label{[kh]} For $x,x'\leq_\mathtt{R} d$, if $h_{z,x,d}=h_{z,x',d}$, for all $z\in W$, then $x=x'$.
\end{enumerate}
\end{corollary}
\begin{proof}
Suppose $\mathbf{K}(y)$ is true. Then Theorem \ref{DEthet} says that $E_y=E_d$, Condition \eqref{E=E} is true, and $\mathbf{K}(d)$ is true. By assumption, we also have $[\mathbf{Kh}(d)]$ which is equivalent to Condition \eqref{[kh]} by \eqref{gdim=h}.

For the other direction, suppose Conditions \eqref{E=E}, \eqref{[kh]} are true. Theorem \ref{DEthet} and Condition \eqref{E=E} implies $\mathbf{K}(d)=\mathbf{K}(y)$. Since Condition \eqref{[kh]} implies $\mathbf{K}(d)$ by assumption, it follows that $\mathbf{K}(y)$ is true.
\end{proof}

\subsection{Extra results in classical types}
In this subsection, unless explicitly stated otherwise, 
we assume $\mathfrak{g}$ to be 
of (classical) type $A$, $B$, $C$, or $D$. 

The asymptotic algebra endows each diagonal $\mathtt{H}$-cell 
$\mathcal{H}$ with the structure of an abelian group.
This group depends only on the two-sided cell containing
$\mathcal{H}$ and is isomorphic to $(\Z/2\Z)^k$, for some 
$k\in\mathbb N$, see \cite{luneq}.
The off-diagonal $\mathtt{H}$-cells in the same two-sided cell 
have order $2^l$, for $l\leq k$ (see \cite[Section~7]{MMMTZ} 
for more information). If $\mathcal{H}$ is an $\mathtt{H}$-cell,
$\mathcal{H}_l$ is the diagonal $\mathtt{H}$-cell in the left cell
of $\mathcal{H}$, and $\mathcal{H}_r$ is the diagonal 
$\mathtt{H}$-cell in the right cell of $\mathcal{H}$,
then the asymptotic algebra endows $\mathcal{H}$ with the
structure of a transitive $\mathcal{H}_l$-set via right 
multiplication, and with the
structure of a transitive $\mathcal{H}_r$-set via left 
multiplication. In particular, we can speak about stabilizers 
of elements from $\mathcal{H}$ in both, $\mathcal{H}_l$
and $\mathcal{H}_r$. This combinatorics allows us to 
formulate a necessary condition for $\mathbf{Kh}(y)$,
and thus also for $\mathbf{K}(y)$.

\begin{proposition}\label{J maxH}
Let $\mathfrak{g}$ be of classical type and $y\in W$.
Assume that the $\mathtt{H}$-cell of $y$ does not have 
the maximal cardinality among the $\mathtt{H}$-cells
inside the two-sided cell of $y$.
Then both $\mathbf{Kh}(y)$ and $\mathbf{K}(y)$ are false.
\end{proposition}

\begin{proof}
Let $y\in\mathcal{H}$ and $\mathcal{H}_l$
and $\mathcal{H}_r$ be as above, in particular,
$|\mathcal{H}|<|\mathcal{H}_l|=|\mathcal{H}_r|$.
Let $d\in \mathcal{H}_l$ be a Duflo element.
As $\mathbf{K}(y)$ implies $\mathbf{Kh}(y)$ by 
Proposition~\ref{K to JKM}, we only need to show that
$\mathbf{Kh}(y)$ is false. Taking \eqref{cz} into account, 
it is enough to find
$z\in \mathcal{H}$ such that $z\neq d$ and $t_yt_d=t_yt_z$.
As $|\mathcal{H}|<|\mathcal{H}_l|$, the stabilizer of 
$y$ in $\mathcal{H}_l$ is non-trivial and hence we can take
as $z$ any element from this stabilizer different from $d$.
The claim follows.
\end{proof}

\begin{lemma}\label{decyy}
Suppose $W$ is of classical type, $\mathcal{L}$ is a left cell 
in $W$ and $y\in \mathcal{L}$. Set $\mathcal{H} = \mathcal{L} 
\cap \mathcal{L}\inv$. Then 
\[\theta_{y\inv}\theta_{y} \cong \bigoplus_{z\in \operatorname{Stab}_\mathcal{H}(y)}\theta_z\langle \mathbf{a}(y)\rangle \oplus \theta,\]
where each summand of $\theta$ is of the form 
$\theta_w\langle a\rangle$, where $a < \mathbf{a}(w)$.
\end{lemma}

\begin{proof}
The element $t_yt_{y\inv}$ of the asymptotic ring is
a linear combination of $t_u$, for $u\in\mathcal{H}$,
and has the same left stabilizer as $y$. 
Therefore $t_yt_{y\inv}$ must be a scalar multiple of 
$\displaystyle\sum_{z\in\operatorname{Stab}_\mathcal{H}(y)} t_z$.
But the multiplicity of $\theta_d$, where $d\in\mathcal{H}$
is the Duflo element, in $\theta_{y\inv}\theta_y$ 
is one. Therefore the scalar in question is one
which implies our claim.
\end{proof}

\begin{lemma}\label{decyLy}
Suppose $W$ is of classical type, $\mathcal{L}$ is a left cell 
in $W$ and $y\in \mathcal{L}$. Set $\mathcal{H} = \mathcal{L} 
\cap \mathcal{L}\inv$ and let $d\in\mathcal{H}$ be the Duflo element. 
Then
\[\theta_{y\inv}L_{y} \cong 
\bigoplus_{z\in \operatorname{Stab}_\mathcal{H}(y)}\theta_zL_d. \]
In particular, $\theta_{y\inv} L_{y}$ is indecomposable if and only if 
$\mathcal{H}$ is of maximal cardinality in 
its two-sided cell.
\end{lemma}

\begin{proof}
This follows from Proposition~\ref{b=c} and Lemma~\ref{decyy}.
\end{proof}

\begin{lemma}\label{zeroend}
Let $W$ be of any type and $x,y\in W$. Then we have
\[\operatorname{hom}(\theta_xL_y,\theta_xL_y)=
\bigoplus_{w\in \mathcal{L}\cap \mathcal{L}\inv}
\mathrm{hom}(\theta_w L_y, 
L_y\langle\mathbf{a}(x)\rangle)^{\oplus \gamma_{x,x\inv,w\inv}},
\] 
where $\mathcal L$ is the left cell of $x$.
\end{lemma}

\begin{proof}
This follows from Theorem~\ref{thm-kmm-original} and the definition
of the asymptotic ring by adjunction.
\end{proof}

\begin{lemma}\label{xtoy} 
Let $W$ be of any type, $x,y\in W$ and $d$ the Duflo element such that
$d\sim_\mathtt{R} x$. Assume that $\mathbf{KMM}(d,y)$ is true.
Then $\mathbf{KM}({x\inv},x)$ 
implies $\mathbf{KM}(x,y)$.
\end{lemma}

\begin{proof}
If $\theta_{x\inv}L_x$ is indecomposable, then Proposition~\ref{b=c} implies $t_{x}t_{x\inv}=t_d$.
Hence, by Lemma~\ref{zeroend}, $\mathrm{end}(\theta_xL_y)$ 
is isomorphic to $\hom(\theta_dL_y,L_y\langle\mathbf a(x)\rangle)$
and thus has dimension one by $\mathbf{KMM}(d,y)$.
This implies the claim.
\end{proof}

\begin{lemma}\label{ytox1}
Let $W$ be of any type, $y,z\in W$ and $d$ be the Duflo element such that
$d\sim_\mathtt{R} w_0y\inv$. Suppose $\mathbf{KMM}(d,z)$ and $\mathbf{KM}(y\inv,y)$ are true.
Then $\mathbf{KM}(w_0z\inv,y)$ is true.
\end{lemma}

\begin{proof}
The indecomposability of $\theta_{y\inv} L_y$ is, by 
Proposition~\ref{b=c}, equivalent to 
$t_{y}t_{y\inv}=t_{d'}$, where $d'$ is an appropriate Duflo element.
As we saw in the proof of Theorem~\ref{theorem1},
this, in turn, is equivalent to 
$t_{y\inv}t_{y}=t_{d''}$, where $d''$ is an appropriate Duflo element.
The latter is equivalent to the
indecomposability of $\theta_{y}L_{y\inv}$.
By Koszul-Ringel duality, we get that 
$\theta_{y w_0}L_{w_0y\inv}$ is indecomposable. 
Then $\mathbf{KMM}(d,z)$ and Lemma~\ref{xtoy} implies 
that $\theta_{w_0y\inv} L_z$ is indecomposable. Then, by Koszul-Ringel duality again, $\theta_{z\inv w_0}L_{w_0y w_0}$ is indecomposable. Since conjugation by $w_0$ corresponds to an automorphism of the Dynkin diagram, 
we obtain that $\theta_{w_0z\inv} L_{y}$ is indecomposable. 
\end{proof}

We can now establish a general sufficient condition for
equivalence of $\mathbf{K}(y)$ is equivalent to $\mathbf{Kh}(y)$.

\begin{proposition}\label{K=J ABCD}
Let $\mathfrak{g}$ be of classical type, $y\in W$
and $d$ be the Duflo element such that $d\sim_\mathtt{R} w_0y\inv$.
Assume that $\mathbf{KMM}(d,z)$ is true, for all $z\in W$. 
Then $\mathbf{K}(y)$ is equivalent to $\mathbf{Kh}(y)$.
\end{proposition}

\begin{proof}
By Theorem~\ref{theorem1}, it is enough to show that $\mathbf{Kh}(y)$ implies $\mathbf{KM}(*,y)$. The latter assertion
would follow from Lemma~\ref{ytox1} provided that  
we can show that $\theta_{y\inv}L_y$ is indecomposable.
Given $\mathbf{Kh}(y)$, from Proposition~\ref{J maxH}
it follows that the $\mathtt{H}$-cell of $y$ is of maximal
cardinality inside the $\mathtt{J}$-cell of $y$.
Hence indecomposability of $\theta_{y\inv}L_y$
follows from Lemma~\ref{decyLy}.
\end{proof}

In this subsection, we do need the assumption on 
$\mathfrak{g}$ to be of classical type (compare with Example~\ref{G2}).

\section{Further discussion and speculation 
on $\mathbf{KMM}$}\label{s-kmm+}

\subsection{A counterexample}\label{s-kmm+.0}
In this subsection we give an  example in which 
$\mathbf{KMM}$ is false.

Let $(W,S)$ be of type $B_3$. We label the simple reflections in the following way:
\begin{tikzpicture}[scale=0.4,baseline=-2]
\draw (0 cm,0) -- (2 cm,0);
\draw (2 cm, 0.1 cm) -- +(2 cm,0);
\draw (2 cm, -0.1 cm) -- +(2 cm,0);
\draw[fill=white] (0 cm, 0 cm) circle (.15cm) node[above=1pt]{\scriptsize $1$};
\draw[fill=white] (2 cm, 0 cm) circle (.15cm) node[above=1pt]{\scriptsize $2$};
\draw[fill=white] (4 cm, 0 cm) circle (.15cm) node[above=1pt]{\scriptsize $3$};
\end{tikzpicture}.
Consider $y=231232$ and a Duflo element $x=2312312$.
We have $\mathbf{a}(x)=3$.
One can check (using a computer) that the graded 
composition factors of $\theta_xL_y$ is as given in Figure~\ref{figaa2}
and directly see that $\dim\hom(\theta_x L_y,L_y\langle \mathbf{a}(x)\rangle) = 2$. 
We note also that we have a decomposition $\theta_xL_y \cong \theta_{12312}L_{23123121} \oplus \theta_{12312}L_{123121}$,
supporting Conjecture~\ref{conj3prime}.

\begin{figure}
\begin{align*}
\text{degree}\colon& \text{composition factors in this degree:} \\
\mathbf{-3}\colon& 123121,\ 23123121,\ \mathbf{231232^{\oplus 2}} \\
-2\colon& 12321,\ 23123^{\oplus 2},\ 121,\ 2312312^{\oplus 2},\ 2312321^{\oplus 4},\ 23121,\ 12312^{\oplus 2} \\
-1\colon& 23123121^{\oplus 4},\ 1231^{\oplus 3},\ 2312^{\oplus 2},\ 2321,\ 12,\ 231231^{\oplus 3}, 123121^{\oplus 4},\ 231232^{\oplus 4},\ 21 \\
0\colon& 12321^{\oplus 3},\ 23123^{\oplus 4},\ 121^{\oplus 2},\ 1,\ 123,\ 2312312^{\oplus 4},\ 2312321^{\oplus 8},\ 231^{\oplus 2},\ 23121^{\oplus 2},\ 12312^{\oplus 4} \\
1\colon& 23123121^{\oplus 4},\ 1231^{\oplus 3},\ 2312^{\oplus 2},\ 2321,\ 12,\ 231231^{\oplus 3}, 123121^{\oplus 4},\ 231232^{\oplus 4},\ 21 \\
2\colon& 12321,\ 23123^{\oplus 2},\ 121,\ 2312312^{\oplus 2},\ 2312321^{\oplus 4},\ 23121,\ 12312^{\oplus 2} \\
3\colon& 123121,\ 23123121,\ 231232^{\oplus 2}
\end{align*}
\caption{The composition factors of $\theta_xL_y$ 
from Subsection~\ref{s-kmm+.0}}
\label{figaa2}
\end{figure}

\subsection{A homological approach to $\mathbf{KMM}$}\label{s-kmm+.1}
Here is a general criterion for $\mathbf{KMM}$
given in homological terms.

\begin{proposition}\label{prophom1}
Let $x,y\in W$ and $\mathcal{H}=\mathcal{L}\cap\mathcal{L}\inv$,
where $\mathcal{L}$ is the left cell of $w_0y$. 
Then $\mathbf{KMM}(x,y)$ is equivalent to 
\begin{equation}\label{prophom1-eq2}
\sum_{z\in \mathcal{H}}
\gamma_{y\inv w_0,w_0y,z\inv}\cdot \dim\mathrm{ext}^{\mathbf{a}(x)}(\theta_{z}L_{w_0x\inv},
L_{w_0}\langle\mathbf{a}(y\inv w_0)-\mathbf{a}(x)\rangle)\leq 1.
\end{equation}
\end{proposition}

We note that, in type $A$, the formula \eqref{prophom1-eq2} simplifies to\begin{equation}\label{kmmtypea}
\dim\mathrm{ext}^{\mathbf{a}(x)}(\theta_{d}L_{w_0x\inv},
L_{w_0}\langle\mathbf{a}(y\inv w_0)-\mathbf{a}(x)\rangle)\leq 1,
\end{equation}
where $d\in \mathcal{H}$ is the Duflo element (in fact, 
$\mathcal{H}=\{d\}$ in this case). In classical types,
the formula \eqref{prophom1-eq2} reads
\begin{displaymath}
\sum_{z\in \mathrm{Stab}_{\mathcal{H}}(y\inv w_0)}
\dim\mathrm{ext}^{\mathbf{a}(x)}(\theta_{d}L_{w_0x\inv},
L_{w_0}\langle\mathbf{a}(y\inv w_0)-\mathbf{a}(x)\rangle)\leq 1.
\end{displaymath}

\begin{proof}[Proof of Proposition~\ref{prophom1}.]
We realize $\theta_xL_y$ as a linear complex
of tilting modules. Using Koszul-Ringel duality and adjunction, 
we have
\begin{equation}\label{prophom1-eq3}
    \begin{split}
        \hom(\theta_x L_y,L_y\langle \mathbf{a}(x)\rangle)&\cong \ext^{\mathbf{a}(x)}(\theta_{y\inv w_0} L_{w_0x\inv},\theta_{y\inv w_0}L_{w_0}\langle -\mathbf{a}(x)\rangle)\\
        &\cong \ext^{\mathbf{a}(x)}(\theta_{w_0y}\theta_{y\inv w_0}L_{w_0x\inv},L_{w_0}\langle -\mathbf{a}(x)\rangle).
    \end{split}
\end{equation}
For $z\in\mathcal{H}$, the composition
$\theta_{w_0y}\theta_{y\inv w_0}$ contains
$\theta_z\langle-\mathbf{a}(y\inv w_0)\rangle$
with multiplicity $\gamma_{y\inv w_0,w_0y,z\inv}$. All other summands 
of this composition are of the form
$\theta_w\langle-a\rangle$, where $w\geq_{\mathtt{J}} d$
and $a<\mathbf{a}(w)$. By \eqref{prophom1-eq2}, the contribution of all
$\theta_z\langle-\mathbf{a}(y\inv w_0)\rangle$ to the right hand side of
\eqref{prophom1-eq3} is at most one-dimensional. So, to complete
the proof we only need to show that the contribution of
all other summands of $\theta_{w_0y}\theta_{y\inv w_0}$ 
to the right hand side of \eqref{prophom1-eq3} is zero.

To this end, fix $\theta_w\langle-a\rangle$ as above
and consider the linear complex $\mathcal{T}_{\bullet}$
of tilting modules representing $\theta_wL_{w_0x\inv}$.
By Koszul-Ringel duality, each indecomposable summand of 
$\mathcal{T}_{-\mathbf{a}(x)}$ is of the form
$T_u\langle-\mathbf{a}(x)\rangle$, for some $u\in W$ such that
$u\leq_{\mathtt{J}}w_0w\leq_{\mathtt{J}}w_0d$.

For such $u$, we claim that the condition 
$\mathrm{hom}(T_u,L_{w_0}\langle b\rangle)\neq 0$
necessarily implies $b\geq \mathbf{a}(w_0u)=\mathbf{a}(d)$.
Indeed, we have $\mathrm{hom}(T_u,L_{w_0}\langle b\rangle)=
\mathrm{hom}(T_u,T_{w_0}\langle b\rangle)$ which, in turn, equals
$\mathrm{hom}(P_{w_0u},P_{e}\langle b\rangle)$ by 
Soergel's character formula for tilting modules, see
\cite[Theorem~6.7]{So2}. Now our $b\geq \mathbf{a}(w_0u)$ 
follows from \eqref{eq9}.

For such $b$, we have $b-a>\mathbf{a}(y\inv w_0)-\mathbf{a}(x)$
since $-a>-\mathbf{a}(w)$. This implies that the contribution of
$\theta_w\langle-a\rangle$
to the right hand side of \eqref{prophom1-eq3} is zero
and completes the proof.
\end{proof}

\subsection{$\mathbf{KMM}$ via twisting 
functors}\label{s-kmm+.3}

Consider the full twisting functor $\mathbf{T}_{w_0}$, see \cite{AS}.
The following statement provides a reformulation of 
$\mathbf{KMM}$ in terms of the cohomology of the derived twisting
functor $\mathscr{L}\mathbf{T}_{w_0}$ evaluated at a simple module.

\begin{proposition}\label{kmmtwist}
Let $x,y\in W$ and $\mathcal{H}=\mathcal{L}\cap\mathcal{L}\inv$,
where $\mathcal{L}$ is the left cell of $w_0 y$. 
Then $\mathbf{KMM}(x,y)$ is equivalent to 
\begin{equation}\label{kmmtwist-eq3}
\sum_{z\in \mathcal{H}}
\gamma_{y\inv w_0,w_0 y,z\inv}\cdot 
[\mathscr{L}_{\mathbf{a}(x)}\mathbf{T}_{w_0}
(L_{w_0x\inv}):L_{z\inv}
\langle\mathbf{a}(y\inv w_0)-\mathbf{a}(x)\rangle]\leq 1.
\end{equation}
\end{proposition}

In type $A$, the formula \eqref{kmmtwist-eq3} simplifies to
\begin{displaymath}
[\mathscr{L}_{\mathbf{a}(x)}\mathbf{T}_{w_0}
(L_{w_0x\inv}):L_d
\langle\mathbf{a}(y\inv w_0)-\mathbf{a}(x)\rangle]\leq 1,
\end{displaymath}
where $d$ is the Duflo element in 
$\mathcal{H}$. 
In classical types, the formula \eqref{kmmtwist-eq3} reads
\begin{displaymath}
\sum_{z\in \mathrm{Stab}_{\mathcal{H}}(y\inv w_0)}
[\mathscr{L}_{\mathbf{a}(x)}\mathbf{T}_{w_0}
(L_{w_0x\inv}):L_{z\inv}
\langle\mathbf{a}(y\inv w_0)-\mathbf{a}(x)\rangle]\leq 1.
\end{displaymath}

\begin{proof}[Proof of Proposition~\ref{kmmtwist}.]
In this proof we use standard properties of
twisting functors, see \cite{AS} for details.
To start with, we note that $\mathscr{L}\mathbf{T}_{w_0}$ 
is an auto-equivalence of the bounded
derived category of $\mathcal{O}_0$. Let us apply 
this auto-equivalence to the left hand side of \eqref{prophom1-eq2}. 
In the second argument of the extension space,
we get $\mathscr{L}\mathbf{T}_{w_0}(T_{w_0})=I_e$.
In the first argument of the extension space, we note that
$\mathscr{L}\mathbf{T}_{w_0}$ commutes with projective functors,
so we can move $\theta_z$ out. By adjunction, we move
$\theta_z$ over to the second argument obtaining
$\theta_{z\inv} I_e\cong I_{z\inv}$. The leftover in
the first argument is $\mathscr{L}\mathbf{T}_{w_0}
(L_{w_0x\inv})$. Evaluating at the correct degree of the
extension and noting that homomorphisms to $I_{z\inv}$
give exactly the composition multiplicity of 
$L_{z\inv}$ for the homology, we obtain \eqref{kmmtwist-eq3}.
\end{proof}

\section{Small rank results}\label{section:smallrank}

The results in this paper, together with (computer-assisted computations for) Kazhdan-Lusztig combinatorics, enable us to determine $\mathbf{K}(y)$, $\mathbf{KM}(*,y)$, and $\mathbf{Kh}(y)$ in many cases. 
We present some of the results in this section.

In computer-assisted calculations, SageMath v.9.0 has been used.

\subsection{Type $A$}

It is verified in \cite{CMZ} that $\mathbf{KM}(x,y)$ is true for all $x,y$ in type $A_n$ for $n\leq 5$. 
Therefore, in this case we only need to determine either $\mathbf{Kh}(y)$ or $\mathbf{K}(y)$, which are equivalent by Theorem \ref{theorem1}.

In \cite{KhM,Kh},  Kostant's problem is solved for simple highest weight modules in $\mathcal{O}_0$, for $\mathfrak{sl}_n$, $n \leq 5$. For $\mathfrak{sl}_6$, i.e., type $A_5$, out of $76$ Duflo elements, $47$ have positive answer to Kostant's problem, $20$ have negative, and the following $9$ Duflo elements were left as an open problem (the notation for simple reflections is analogous to the one in Example \ref{ex9}):
\begin{equation}
\label{eq:KA5open}
\begin{tabular}{lll}
$23432$,   &&      $4523412$, \\
$234312$,  &&      $12342321$, \\
$452342$,  &&      $23454232$, \\
$234512342312$, && $3451234231$, \\
 &&                $2345412312$.                    
\end{tabular}
\end{equation}
We can now solve these remaining cases.

For all $y$ in the first column in \eqref{eq:KA5open}, one can check (by computer) that $[\mathbf{Kh}](y)$ holds, and therefore $\mathbf{Kh}(y)$ holds. From \cite{CMZ} and Theorem \ref{theorem1} it follows that $\mathbf{K}(y)$ is true.

For the second column in \eqref{eq:KA5open}, it is enough to consider only $x := 4523412$, $y := 12342321$ and $z := 3451234231$, because of the symmetry of the Dynkin diagram. By \cite[Proposition 46.b)]{CMZ}, we have
\[ \theta_{23} L_x \cong \theta_{3} L_{x3} \cong \theta_{43} L_x  \qquad \text{and} \qquad \theta_{32} L_z \cong \theta_{2} L_{z2} \cong \theta_{12} L_z.  \]
Therefore both $\mathbf{Kh}(x)$ and $\mathbf{Kh}(z)$ fail. So again, using \cite{CMZ} and
Theorem~\ref{theorem1}, we conclude that neither $\mathbf{K}(x)$ nor $\mathbf{K}(z)$ holds.

A calculation shows that $\theta_{45342} \cong \theta_{342}\theta_{45}$. One can check that $\theta_{45} L_y$ has height $1$. Moreover, its top and socle are simple, consisting of $L_{y'}$ with $y':=123452321$, and $L_y$ in the middle. Since $342 \not\leq_\mathtt{R} (y')^{-1}$, it follows that
\[ \theta_{45342} L_y \cong \theta_{342} L_y. \]
Therefore $\mathbf{Kh}(y)$ does not hold, and so $\mathbf{K}(y)$ also does not hold. Since the property $\mathbf{K}$ in type $A$ is invariant for the left cells, we have:

\begin{corollary}
\label{corollary:kostantA5}
Kostant's problem has a positive solution, for 
a simple highest weight module $L_w$ for $\mathfrak{sl}_6$, 
if and only if $w$ does not belong to the left cells 
containing one of the following 25 Duflo elements:
\[ \begin{tabular}{lllll}
$31$, & $531$, & $45341$, & $512321$, & $23454232$, \\
$42$, & $3431$, & $52312$, & $3453431$, & $34541231$, \\
$53$, & $4121$, & $234232$, & $4523412$, & $2345412312$, \\
$232$, & $4542$, & $345431$, & $5123121$, & $3451234231$, \\
$343$, & $5232$, & $454121$, & $12342321$, & $12345343121$.
\end{tabular}\]
\end{corollary}

\subsection{Types $BCD$}

We completely determine $\mathbf{K}(w)$, $\mathbf{KM}(\ast,w)$ and $\mathbf{Kh}(w)$, for each $w\in W$ in types $B_3$ and $D_4$.
We also determine $\mathbf{K}(w)$ and $\mathbf{Kh}(w)$ completely in type $B_4$.
In all above examples, we have $\mathbf{Kh}(w)\Leftrightarrow\mathbf{KM}(\ast,w)$, supporting Conjecture \ref{kconj-intro}. 
We provide below the results without details for type $B_3$ and $D_4$.

A good way to present these results is to mark the $\mathtt{H}$-cells in $W$ in the following way
(note that, since $W$ is of classical type, 
by Proposition \ref{prop-n942} and Lemma \ref{decyLy}, the function $\mathbf{K}(w)$ (resp., $\mathbf{KM}(\ast,w)$ and $\mathbf{Kh}(w)$) has the same value for $w$ in the same $\mathtt{H}$-cell) :

\definecolor{color_K}{gray}{0.7}
\definecolor{color_KM_notKh}{gray}{0.85}
\definecolor{color_checkKM_notKh}{rgb}{1, 0, 0}

The elements $w$ in the $\mathtt{H}$-cells colored in
\begin{itemize}
    \item \colorbox{color_K}{this} color satisfy $\mathbf{K}(w)$, $\mathbf{KM}(\ast,w)$ and $\mathbf{Kh}(w)$,
    \item \colorbox{color_KM_notKh}{this} color satisfy $\mathbf{KM}(\ast,w)$ and do not satisfy  $\mathbf{Kh}(w)$, thus do not satisfy $\mathbf{K}(w)$, 
    \item white do not satisfy any of the above properties.
\end{itemize} 
Figure~\ref{figure:cellsB3} presents type $B_3$ results and Figure~\ref{figure:cellsD4} gives type $D_4$ results.

\begin{figure}[ht]
\begin{tabular}{|l|}
\hline
\cellcolor{color_K}$e$ \\ \hline
\end{tabular}
\hskip .75cm
\begin{tabular}{|l|l|l|}
\hline
$2312$ & $312$ & $32312$ \\ \hline
$231$ & $31$ & $3231$ \\ \hline
\cellcolor{color_K}$23123$ & \cellcolor{color_K}$3123$ & \cellcolor{color_K}$323123$ \\ \hline
\end{tabular}
\hskip .75cm
\begin{tabular}{|l|l|l|}
\hline
\cellcolor{color_K}$121$ & \cellcolor{color_K}$23121$ & \cellcolor{color_K}$3121$ \\ \hline
\cellcolor{color_K}$12312$ & \cellcolor{color_K}$2312312$ & \cellcolor{color_K}$312312$ \\ \hline
$1231$ & $231231$ & $31231$ \\ \hline
\end{tabular}
\vskip .5cm
\begin{tabular}{|l|l|l|}
\hline
\cellcolor{color_K}\begin{tabular}{@{}l@{}}$1$ \\ $12321$\end{tabular} & \cellcolor{color_K}\begin{tabular}{@{}l@{}}$21$ \\ $2321$\end{tabular} & $321$ \\ \hline
\cellcolor{color_K}\begin{tabular}{@{}l@{}}$12$ \\ $1232$\end{tabular} & \cellcolor{color_K}\begin{tabular}{@{}l@{}}$2$ \\ $232$\end{tabular} & $32$ \\ \hline
$123$ & $23$ & \cellcolor{color_K}\begin{tabular}{@{}l@{}}$3$ \\ $323$\end{tabular} \\ \hline
\end{tabular}
\hskip .75cm
\begin{tabular}{|l|l|l|}
\hline
\cellcolor{color_K}\begin{tabular}{@{}l@{}}$23123121$ \\ $123121$\end{tabular} & $3123121$ & $323121$ \\ \hline
$2312321$ & \cellcolor{color_K}\begin{tabular}{@{}l@{}}$312321$ \\ $32312321$\end{tabular} & \cellcolor{color_K}\begin{tabular}{@{}l@{}}$32321$ \\ $3231231$\end{tabular} \\ \hline
$231232$ & \cellcolor{color_K}\begin{tabular}{@{}l@{}}$31232$ \\ $3231232$\end{tabular} & \cellcolor{color_K}\begin{tabular}{@{}l@{}}$3232$ \\ $32312312$\end{tabular} \\ \hline
\end{tabular}
\hskip .75cm
\begin{tabular}{|l|}
\hline
\cellcolor{color_K}$323123121$ \\ \hline
\end{tabular}
\caption[Cells in type $B_3$]{Cells in type $B_3$ with the labeling
\begin{tikzpicture}[scale=0.4,baseline=-2]
\protect\draw (0 cm,0) -- (2 cm,0);
\protect\draw (2 cm, 0.1 cm) -- +(2 cm,0);
\protect\draw (2 cm, -0.1 cm) -- +(2 cm,0);
\protect\draw[fill=white] (0 cm, 0 cm) circle (.15cm) node[above=1pt]{\scriptsize $1$};
\protect\draw[fill=white] (2 cm, 0 cm) circle (.15cm) node[above=1pt]{\scriptsize $2$};
\protect\draw[fill=white] (4 cm, 0 cm) circle (.15cm) node[above=1pt]{\scriptsize $3$};
\end{tikzpicture}. 
Rows are left, and columns are right cells. Duflo elements are the top elements in the diagonal blocks.}
\label{figure:cellsB3}
\end{figure}

The following example is potentially related to Conjecture \ref{kconj-intro} \ref{kconj-2-intro}$\Leftrightarrow$\ref{kconj-3-intro}:
\begin{example}
Let $W$ be of type $B_3$ as in Figure \ref{figure:cellsB3}. 
In $\mathrm{Gr}(  \cO_0^\Z)$ we have the following equalities:
\begin{align*}
[\theta_{123}L_{121}] &= [\theta_{3}L_{231}] +[\theta_{3231}L_{1231}] \\
&= [\theta_{3}L_{23123}] +[\theta_{323123}L_{1231}].
\end{align*}
However, as shown in Figure \ref{figure:cellsB3}, the object $\theta_{123}L_{121}$ is indecomposable.
\end{example}

\begin{figure}
\begin{tabular}{|l|}
\hline
\cellcolor{color_K}$e$ \\ \hline
\end{tabular}
\hskip .75cm
\begin{tabular}{|l|l|l|l|}
\hline
\cellcolor{color_K}$1$ & \cellcolor{color_K}$21$ & \cellcolor{color_K}$321$ & \cellcolor{color_K}$421$ \\ \hline
\cellcolor{color_K}$12$ & \cellcolor{color_K}$2$ & \cellcolor{color_K}$32$ & \cellcolor{color_K}$42$ \\ \hline
\cellcolor{color_K}$123$ & \cellcolor{color_K}$23$ & \cellcolor{color_K}$3$ & \cellcolor{color_K}$423$ \\ \hline
\cellcolor{color_K}$124$ & \cellcolor{color_K}$24$ & \cellcolor{color_K}$324$ & \cellcolor{color_K}$4$ \\ \hline
\end{tabular}
\hskip .75cm
\begin{tabular}{|l|l|l|}
\hline
\cellcolor{color_K}$2312$ & \cellcolor{color_K}$312$ & \cellcolor{color_K}$42312$ \\ \hline
\cellcolor{color_KM_notKh}$231$ & \cellcolor{color_KM_notKh}$31$ & \cellcolor{color_KM_notKh}$4231$ \\ \hline
\cellcolor{color_K}$23124$ & \cellcolor{color_K}$3124$ & \cellcolor{color_K}$423124$ \\ \hline
\end{tabular}
\vskip .5cm
\begin{tabular}{|l|l|l|}
\hline
\cellcolor{color_K}$2412$ & \cellcolor{color_K}$32412$ & \cellcolor{color_K}$412$ \\ \hline
\cellcolor{color_K}$24123$ & \cellcolor{color_K}$324123$ & \cellcolor{color_K}$4123$ \\ \hline
\cellcolor{color_KM_notKh}$241$ & \cellcolor{color_KM_notKh}$3241$ & \cellcolor{color_KM_notKh}$41$ \\ \hline
\end{tabular}
\hskip .75cm
\begin{tabular}{|l|l|l|}
\hline
\cellcolor{color_K}$124321$ & \cellcolor{color_K}$24321$ & \cellcolor{color_K}$4321$ \\ \hline
\cellcolor{color_K}$12432$ & \cellcolor{color_K}$2432$ & \cellcolor{color_K}$432$ \\ \hline
\cellcolor{color_KM_notKh}$1243$ & \cellcolor{color_KM_notKh}$243$ & \cellcolor{color_KM_notKh}$43$ \\ \hline
\end{tabular}
\vskip .5cm
\begin{tabular}{|l|l|l|l|l|l|l|l|}
\hline
\cellcolor{color_K}\begin{tabular}{@{}l@{}}$121$ \\ $1243121$\end{tabular} & \cellcolor{color_K}\begin{tabular}{@{}l@{}}$3121$ \\ $31243121$\end{tabular} & \cellcolor{color_K}\begin{tabular}{@{}l@{}}$4121$ \\ $12423121$\end{tabular} & \cellcolor{color_K}\begin{tabular}{@{}l@{}}$23121$ \\ $3243121$\end{tabular} & \cellcolor{color_K}\begin{tabular}{@{}l@{}}$24121$ \\ $2423121$\end{tabular} & \cellcolor{color_K}\begin{tabular}{@{}l@{}}$423121$ \\ $324121$\end{tabular} & $243121$ & $43121$ \\ \hline
\cellcolor{color_K}\begin{tabular}{@{}l@{}}$1231$ \\ $12412321$\end{tabular} & \cellcolor{color_K}\begin{tabular}{@{}l@{}}$12321$ \\ $312412321$\end{tabular} & \cellcolor{color_K}\begin{tabular}{@{}l@{}}$41231$ \\ $1242321$\end{tabular} & \cellcolor{color_K}\begin{tabular}{@{}l@{}}$2321$ \\ $32412321$\end{tabular} & \cellcolor{color_K}\begin{tabular}{@{}l@{}}$241231$ \\ $242321$\end{tabular} & \cellcolor{color_K}\begin{tabular}{@{}l@{}}$42321$ \\ $3241231$\end{tabular} & $2412321$ & $412321$ \\ \hline
\cellcolor{color_K}\begin{tabular}{@{}l@{}}$1241$ \\ $23124321$\end{tabular} & \cellcolor{color_K}\begin{tabular}{@{}l@{}}$31241$ \\ $3124321$\end{tabular} & \cellcolor{color_K}\begin{tabular}{@{}l@{}}$12421$ \\ $423124321$\end{tabular} & \cellcolor{color_K}\begin{tabular}{@{}l@{}}$324321$ \\ $231241$\end{tabular} & \cellcolor{color_K}\begin{tabular}{@{}l@{}}$2421$ \\ $42312421$\end{tabular} & \cellcolor{color_K}\begin{tabular}{@{}l@{}}$32421$ \\ $4231241$\end{tabular} & $2312421$& $312421$ \\ \hline
\cellcolor{color_K}\begin{tabular}{@{}l@{}}$12312$ \\ $1241232$\end{tabular} & \cellcolor{color_K}\begin{tabular}{@{}l@{}}$1232$ \\ $31241232$\end{tabular} & \cellcolor{color_K}\begin{tabular}{@{}l@{}}$124232$ \\ $412312$\end{tabular} & \cellcolor{color_K}\begin{tabular}{@{}l@{}}$232$ \\ $3241232$\end{tabular} & \cellcolor{color_K}\begin{tabular}{@{}l@{}}$24232$ \\ $2412312$\end{tabular} & \cellcolor{color_K}\begin{tabular}{@{}l@{}}$4232$ \\ $32412312$\end{tabular} & $241232$ & $41232$ \\ \hline
\cellcolor{color_K}\begin{tabular}{@{}l@{}}$12412$ \\ $2312432$\end{tabular} & \cellcolor{color_K}\begin{tabular}{@{}l@{}}$312432$ \\ $312412$\end{tabular} & \cellcolor{color_K}\begin{tabular}{@{}l@{}}$1242$ \\ $42312432$\end{tabular} & \cellcolor{color_K}\begin{tabular}{@{}l@{}}$32432$ \\ $2312412$\end{tabular} & \cellcolor{color_K}\begin{tabular}{@{}l@{}}$242$ \\ $4231242$\end{tabular} & \cellcolor{color_K}\begin{tabular}{@{}l@{}}$3242$ \\ $42312412$\end{tabular} & $231242$ & $31242$ \\ \hline
\cellcolor{color_K}\begin{tabular}{@{}l@{}}$124123$ \\ $231243$\end{tabular} & \cellcolor{color_K}\begin{tabular}{@{}l@{}}$31243$ \\ $3124123$\end{tabular} & \cellcolor{color_K}\begin{tabular}{@{}l@{}}$12423$ \\ $4231243$\end{tabular} & \cellcolor{color_K}\begin{tabular}{@{}l@{}}$3243$ \\ $23124123$\end{tabular} & \cellcolor{color_K}\begin{tabular}{@{}l@{}}$2423$ \\ $42312423$\end{tabular} & \cellcolor{color_K}\begin{tabular}{@{}l@{}}$32423$ \\ $423124123$\end{tabular} & $2312423$ & $312423$ \\ \hline
$124312$ & $3124312$ & $1242312$ & $324312$ & $242312$ & $3242312$ & \cellcolor{color_K}\begin{tabular}{@{}l@{}}$24312$ \\ $231242312$\end{tabular} & \cellcolor{color_K}\begin{tabular}{@{}l@{}}$4312$ \\ $31242312$\end{tabular} \\ \hline
$12431$ & $312431$ & $124231$ & $32431$ & $24231$ & $324231$ & \cellcolor{color_KM_notKh}\begin{tabular}{@{}l@{}}$2431$ \\ $23124231$\end{tabular} &  \cellcolor{color_KM_notKh}\begin{tabular}{@{}l@{}}$431$ \\ $3124231$\end{tabular} \\ \hline
\end{tabular}
\vskip .5cm
\begin{tabular}{|l|l|l|}
\hline
$123121$ & $24123121$ & $4123121$ \\ \hline
\cellcolor{color_K}$23124312$ & \cellcolor{color_K}$4231242312$ & \cellcolor{color_K}$423124312$ \\ \hline
\cellcolor{color_KM_notKh}$2312431$ & \cellcolor{color_KM_notKh}$423124231$ & \cellcolor{color_KM_notKh}$42312431$ \\ \hline
\end{tabular}
\hskip .75cm
\begin{tabular}{|l|l|l|l|}
\hline
\cellcolor{color_K}$23124123121$ & \cellcolor{color_K}$324123121$ & \cellcolor{color_K}$3124123121$ & \cellcolor{color_K}$124123121$ \\ \hline
\cellcolor{color_K}$231241232$ & \cellcolor{color_K}$42312412312$ & \cellcolor{color_K}$4231241232$ & \cellcolor{color_K}$423124232$ \\ \hline
\cellcolor{color_K}$2312412321$ & \cellcolor{color_K}$4231241231$ & \cellcolor{color_K}$42312412321$ & \cellcolor{color_K}$4231242321$ \\ \hline
\cellcolor{color_K}$231243121$ & \cellcolor{color_K}$423124121$ & \cellcolor{color_K}$4231243121$ & \cellcolor{color_K}$42312423121$ \\ \hline
\end{tabular}
\vskip .5cm
\begin{tabular}{|l|l|l|}
\hline
\cellcolor{color_K}$2312423121$ & \cellcolor{color_K}$312423121$ &  \cellcolor{color_K}$32423121$ \\ \hline
\cellcolor{color_KM_notKh}$231242321$ & \cellcolor{color_KM_notKh}$31242321$ &\cellcolor{color_KM_notKh}$3242321$ \\ \hline
$23124232$ & $3124232$ & $324232$ \\ \hline
\end{tabular}
\hskip .75cm
\begin{tabular}{|l|l|l|}
\hline
$124121$ & $23124121$ & $3124121$ \\ \hline
\cellcolor{color_K}\cellcolor{color_K}$12412312$ & \cellcolor{color_K}$2312412312$ & \cellcolor{color_K}$312412312$ \\ \hline
\cellcolor{color_KM_notKh}$1241231$ & \cellcolor{color_KM_notKh}$231241231$ & \cellcolor{color_KM_notKh}$31241231$ \\ \hline
\end{tabular}
\hskip .75cm
\begin{tabular}{|l|}
\hline
\cellcolor{color_K}$423124123121$ \\ \hline
\end{tabular}
\caption[Cells in type $D_4$]{Cells in type $D_4$ with labeling %
\begin{tikzpicture}[scale=0.4,baseline=-3]
\protect\draw (0 cm,0) -- (2 cm,0);
\protect\draw (2 cm,0) -- (4 cm,0.7 cm);
\protect\draw (2 cm,0) -- (4 cm,-0.7 cm);
\protect\draw[fill=white] (0 cm, 0 cm) circle (.15cm) node[above=1pt]{\scriptsize $1$};
\protect\draw[fill=white] (2 cm, 0 cm) circle (.15cm) node[above=1pt]{\scriptsize $2$};
\protect\draw[fill=white] (4 cm, 0.7 cm) circle (.15cm) node[above=1pt]{\scriptsize $4$};
\protect\draw[fill=white] (4 cm, -0.7 cm) circle (.15cm) node[above=1pt]{\scriptsize $3$};
\end{tikzpicture}. Rows are left, and columns are right cells. Duflo elements are the top elements in the diagonal blocks.}
\label{figure:cellsD4}
\end{figure}

\vspace{2mm}

\noindent
H.~K.: Department of Mathematics, Uppsala University, Box. 480,
SE-75106, Uppsala,\\ SWEDEN, email: {\tt hankyung.ko\symbol{64}math.uu.se}

\noindent
V.~M.: Department of Mathematics, Uppsala University, Box. 480,
SE-75106, Uppsala,\\ SWEDEN, email: {\tt mazor\symbol{64}math.uu.se}

\noindent
R.~M.: Department of Mathematics, Uppsala University, Box. 480,
SE-75106, Uppsala,\\ SWEDEN, email: {\tt rafael.mrden\symbol{64}math.uu.se} \\
(On leave from: Faculty of Civil Engineering, University of Zagreb, \\
Fra Andrije Ka\v{c}i\'{c}a-Mio\v{s}i\'{c}a 26, 10000 Zagreb, CROATIA)


\begin{thebibliography}{99999999}
\bibitem[AS]{AS} Andersen, H.; Stroppel, C. Twisting functors on $\mathcal{O}$. 
Represent. Theory {\bf 7} (2003), 681--699. 
\bibitem[BV1]{BV1} Barbasch, D.; Vogan, D. Primitive ideals and orbital 
integrals in complex classical groups. Math. Ann. {\bf 259} (1982), no. 2, 153--199.
\bibitem[BV2]{BV2} Barbasch, D.; Vogan, D. Primitive ideals and orbital 
integrals in complex exceptional groups. J. Algebra {\bf 80} (1983), no. 2, 350--382. 
\bibitem[BB]{BB} Beilinson, A.; Bernstein, J. Localisation de 
$\mathfrak{g}$-modules. C. R. Acad. Sci. Paris Ser. I Math. {\bf 292} (1981), no. 1, 15--18. 
\bibitem[BG]{BG} Bernstein, I. N.; Gelfand, S. I. Tensor products of finite- 
and infinite-dimensional representations of semisimple Lie algebras. Compositio
Math. {\bf 41} (1980), no. 2, 245--285. 
\bibitem[BGG]{BGG} Bernstein, I. N.; Gelfand, I. M.; Gelfand, S. I. 
A certain category of $\mathfrak{g}$-modules. (Russian) Funkcional. Anal. i 
Prilozen. {\bf 10} (1976), no. 2, 1--8.
\bibitem[BK]{BK} Brylinski, J.-L.; Kashiwara, M. Kazhdan-Lusztig conjecture and 
holonomic systems. Invent. Math. {\bf 64} (1981), no. 3, 387--410. 
\bibitem[CM1]{CoM0} Coulembier, K.; Mazorchuk, V. 
Some homological properties of category $\mathcal{O}$. III. 
Adv. Math. {\bf 283} (2015), 204--231.
\bibitem[CM2]{CoM} Coulembier, K.; Mazorchuk, V. Some homological properties 
of category $\mathcal{O}$. IV. Forum Math. {\bf 29} (2017), no. 5, 1083--1124.
\bibitem[CMZ]{CMZ} Coulembier, K.; Mazorchuk, V.; Zhang, X. Indecomposable 
manipulations with simple modules in category $\mathcal{O}$. Math. Res. Lett. 
{\bf 26} (2019), no. 2, 447--499.
\bibitem[EGNO]{EGNO} Etingof, P.; Gelaki, S.; Nikshych, D.; Ostrik, V. 
Tensor categories. Mathematical Surveys and Monographs, {\bf 205}. 
American Mathematical Society, Providence, RI, 2015. xvi+343 pp.
\bibitem[EW]{EW} Elias, B., Williamson, G. The Hodge theory of Soergel bimodules. 
Ann. of Math. (2) {\bf 180} (2014), no. 3, 1089--1136.
\bibitem[GJ]{GJ} Gabber, O.; Joseph, A. On the Bernstein-Gelfand-Gelfand 
resolution and the Duflo sum formula. Compositio Math. {\bf 43} (1981), 
no. 1, 107--131. 
\bibitem[Hu]{Hu} Humphreys, J. E. Representations of semisimple Lie algebras in 
the BGG category $\mathcal{O}$. Graduate Studies in Mathematics, {\bf 94}. 
American Mathematical Society, Providence, RI, 2008. xvi+289 pp.
\bibitem[Ja]{Ja} Jantzen, J. C. Einh{\"u}llende Algebren halbeinfacher 
Lie-Algebren. Ergebnisse der Mathematik und ihrer Grenzgebiete (3), 
{\bf 3}. Springer-Verlag, Berlin, 1983. ii+298 pp. 
\bibitem[Jo]{Jo} Joseph, A. Kostant's problem, Goldie rank and the 
Gel'fand-Kirillov conjecture. Invent. Math. {\bf 56} (1980), no. 3,  191--213.
\bibitem[KL]{KL} Kazhdan, D.; Lusztig, G. Representations of Coxeter groups 
and Hecke algebras. Invent. Math. {\bf 53} (1979), no. 2, 165--184.
\bibitem[Kho]{Kho} Khomenko, O. Categories with projective functors. 
Proc. London Math. Soc. (3) {\bf 90} (2005), no. 3, 711--737.
\bibitem[KMMZ]{KMMZ} Kildetoft, T.; Mackaay, M.; Mazorchuk, V.; 
Zimmermann, J. Simple transitive $2$-representations of small quotients 
of Soergel bimodules. Trans. Amer. Math. Soc. {\bf 371} (2019), no. 8, 5551--5590.
\bibitem[KM]{KM} Kildetoft, T.; Mazorchuk, V. Parabolic projective 
functors in type $A$. Adv. Math. {\bf 301} (2016), 785--803.
\bibitem[KM26]{CA26} Ko, H., Mazorchuk, V.
Addendum and Corrigendum to: Some homological properties of category $\mathcal{O}$, V.
Preprint arXiv:2611.12892.
\bibitem[Kh]{Kh} K{\aa}hrstr{\"o}m, J. Kostant's problem and parabolic 
subgroups. Glasg. Math. J. {\bf 52} (2010), no. 1, 19--32.
\bibitem[KhM]{KhM} K{\aa}hrstr{\"o}m, J.; Mazorchuk, V. A new approach to 
Kostant's problem. Algebra Number Theory {\bf 4} (2010), no. 3, 231--254.
\bibitem[Lu1]{Lu1} Lusztig, G. Cells in affine Weyl groups, in: Algebraic 
Groups and Related Topics, in: Adv. Stud. Pure Math., vol. {\bf 6}, 
North-Holland, Amsterdam, 1985, pp. 255--287.
\bibitem[Lu2]{Lu2} Lusztig, G. Cells in affine Weyl groups. II, J. Algebra 
{\bf 109} (2) (1987) 536--548.
\bibitem[Lu3]{luneq} Lusztig, G. Hecke algebras with unequal parameters, CRM Monograph Series, {\bf 18}. American Mathematical Society, Providence, RI (2003), vi+136 pp. 
\bibitem[Ma1]{Ma1} Mazorchuk, V. A twisted approach to Kostant's problem. 
Glasg. Math. J. {\bf 47} (2005), no. 3, 549--561.
\bibitem[Ma2]{Ma0} Mazorchuk, V. Some homological properties of the 
category $\mathcal{O}$. Pacific J. Math. {\bf 232} (2) (2007) 313--341. 
\bibitem[Ma3]{Ma2} Mazorchuk, V. Some homological properties of the 
category $\mathcal{O}$. II. Represent. Theory {\bf 14} (2010), 249--263.
\bibitem[MM]{MM0} Mazorchuk, V.; Miemietz, V. Serre functors for Lie 
algebras and superalgebras. Ann. Inst. Fourier (Grenoble) {\bf 62} 
(2012), no. 1, 47--75. 
\bibitem[MM1]{MM1} Mazorchuk, V.; Miemietz, V. Cell 2-representations of 
finitary 2-categories. Compos. Math. {\bf 147} (2011), no. 5, 1519--1545.
\bibitem[MM3]{MM3} Mazorchuk, V.; Miemietz, V. Endomorphisms of cell 
$2$-representations. Int. Math. Res. Not. IMRN {\bf 2016}, no. 24, 7471--7498.
\bibitem[MMMT]{MMMT} Mackaay, M.; Mazorchuk, V.; Miemietz, V.; Tubbenhauer, D. Simple transitive 2-representations via 
(co-)algebra 1-morphisms. Indiana Univ. Math. J. {\bf 68} (2019), 
no. 1, 1--33.
\bibitem[MMMTZ]{MMMTZ} M. Mackaay, V. Mazorchuk, V. Miemietz, D. Tubbenhauer, 
and X. Zhang. 2-representations of Soergel bimodules. arXiv:1906.11468
\bibitem[MO]{MO} Mazorchuk, V.; Ovsienko, S. Finitistic dimension of 
properly stratified algebras. Adv. Math. {\bf 186} (2004), no. 1, 251--265. 
\bibitem[MS1]{MS1} Mazorchuk, V.; Stroppel, C. Categorification of (induced) 
cell modules and the rough structure of generalised Verma modules. 
Adv. Math. {\bf 219} (2008), no. 4, 1363--1426.
\bibitem[MS2]{MS2} Mazorchuk, V.; Stroppel, C. Categorification of 
Wedderburn's basis for $\mathbb{C}[S_n]$. Arch. Math. (Basel) 
{\bf 91} (2008), no. 1, 1--11.
\bibitem[Os]{Os} Ostrik, V. Module categories, weak Hopf algebras and modular invariants. Transform. Groups \textbf{8} (2003), 177–206. 
\bibitem[So1]{So} Soergel, W. Kategorie $\mathcal{O}$, perverse Garben 
und Moduln {\"u}ber den Koinvarianten zur Weylgruppe. J. Amer. Math. Soc. 
{\bf 3} (1990), no. 2, 421--445. 
\bibitem[So2]{So2} Soergel, W. Character formulas for tilting modules 
over Kac-Moody algebras. Represent. Theory {\bf 2} (1998), 432--448.
\bibitem[St]{St} Stroppel, C. Category $\mathcal{O}$: gradings and translation 
functors. J. Algebra {\bf 268} (2003), no. 1, 301--326. 
\end{thebibliography}
\end{document}